\documentclass[12pt]{amsart}
\usepackage{amssymb}  
\usepackage{latexsym} 
\usepackage[all]{xy}
\usepackage{url}

\newcommand{\A}{{\mathcal A}}
\newcommand{\adj}{\operatorname{Adj}}
\newcommand{\bG}{{\mathbf G}}
\newcommand{\bH}{{\mathbf H}}
\newcommand{\bS}{{\mathbf S}}
\newcommand{\bTheta}{{\mathbf \Theta}}
\newcommand{\bu}{{\mathbf u}}
\newcommand{\Br}{\operatorname{Br}}
\newcommand{\C}{{\mathcal C}}
\newcommand{\CC}{{\mathbb C}}
\newcommand{\CT}{{\operatorname{CT}}}
\newcommand{\dKum}{\widetilde{\K}}
\newcommand{\divisor}{\operatorname{div}}
\newcommand{\Div}{\operatorname{Div}}
\newcommand{\eps}{{\varepsilon}}
\newcommand{\F}{{\mathbb F}}
\newcommand{\Gal}{\operatorname{Gal}}
\newcommand{\GL}{\operatorname{GL}}
\newcommand{\im}{\operatorname{Im}}
\newcommand{\inv}{\operatorname{inv}}
\newcommand{\isom}{\cong}
\newcommand{\J}{{\mathcal J}}
\newcommand{\Jac}{\operatorname{Jac}}
\newcommand{\K}{{\mathcal K}}
\newcommand{\kbar}{\overline{k}}
\newcommand{\la}{\lambda}
\newcommand{\Map}{\operatorname{Map}}
\newcommand{\Ok}{{{\mathcal O}_k}}
\newcommand{\Pf}{\operatorname{Pf}}
\newcommand{\PGL}{\operatorname{PGL}}
\newcommand{\PO}{\operatorname{PO}}
\newcommand{\PSO}{\operatorname{PSO}}
\newcommand{\Pic}{\operatorname{Pic}}
\newcommand{\PP}{{\mathbb P}}
\newcommand{\pp}{{\mathfrak p}}
\newcommand{\Q}{{\mathbb Q}}
\newcommand{\Qbar}{{\overline{\mathbb Q}}}
\newcommand{\ra}{\longrightarrow}
\newcommand{\rank}{{\operatorname{rank}}}
\newcommand{\Res}{\operatorname{Res}}
\newcommand{\Sha}{\mbox{\wncyr Sh}}

\newcommand{\tr}{\operatorname{Tr}}
\newcommand{\vv}{{\mathbf v}}
\newcommand{\x}{{\mathbf x}}
\newcommand{\y}{{\mathbf y}}
\newcommand{\z}{{\mathbf z}}
\newcommand{\Z}{{\mathbb Z}}

\newfont{\wncyr}{wncyr10 at 12pt}
\newfont{\wncyrten}{wncyr10 at 10pt}

\newenvironment{ProofOf}[1]{\par\noindent{\em Proof of #1.}}%
                       {\hspace*{\fill}\nobreak$\Box$\par\medskip}

\newtheorem{Proposition}{Proposition}[section]
\newtheorem{Theorem}[Proposition]{Theorem}

\newtheorem{Lemma}[Proposition]{Lemma}

\theoremstyle{definition}

\newtheorem{Remark}[Proposition]{Remark}
\newtheorem{Definition}[Proposition]{Definition}
\newtheorem{Example}[Proposition]{Example}

\begin{document}
\date{9th June 2023}
\title[Computing the Cassels-Tate pairing]%
{Computing the Cassels-Tate pairing \\
on the $2$-Selmer group of a genus 2 Jacobian}

\author{Tom~Fisher}
\address{University of Cambridge,
          DPMMS, Centre for Mathematical Sciences,
          Wilberforce Road, Cambridge CB3 0WB, UK}
\email{T.A.Fisher@dpmms.cam.ac.uk}

\author{Jiali~Yan}
\email{jialiyan.lele@gmail.com}
        

\renewcommand{\baselinestretch}{1.1}
\renewcommand{\arraystretch}{1.3}

\renewcommand{\theenumi}{\roman{enumi}}

\begin{abstract}
  We describe a method for computing the Cassels-Tate pairing on the
  $2$-Selmer group of the Jacobian of a genus $2$ curve.  This can be
  used to improve the upper bound coming from $2$-descent for the rank
  of the group of rational points on the Jacobian.  Our method remains
  practical regardless of the Galois action on the Weierstrass points
  of the genus $2$ curve. It does however depend on being able to find
  a rational point on a certain twisted Kummer surface. The latter
  does not appear to be a severe restriction in practice. In
  particular, we have used our method to unconditionally determine the
  ranks of all genus~$2$ Jacobians in the $L$-functions and modular
  forms database (LMFDB).
\end{abstract}

\maketitle

\enlargethispage{1ex}
\vspace{-2ex}

\tableofcontents

\section{Introduction}

Let $\C$ be a smooth projective curve defined over a number field $k$.
Its Jacobian $\J = \Jac(\C)$ is a principally polarised abelian
variety, also defined over $k$.  We are interested in the group
$\J(k)$ of $k$-rational points on $\J$. By the Mordell-Weil theorem
this is a finitely generated abelian group. We may therefore write
\[\J(k) \isom \Delta \times \Z^r\] where $\Delta$ is a finite abelian
group, called the torsion subgroup, and $r$ is a non-negative integer,
called the rank.

We may compute an upper bound for the rank by carrying out a
$2$-descent.  Specifically, there is a short exact sequence of
$\F_2$-vector spaces
\begin{equation}
\label{exseq-2desc}
0 \to \J(k)/2 \J(k) \to S^{(2)}(\J/k) \to \Sha(\J/k)[2] \to 0
\end{equation}
where the $2$-Selmer group $S^{(2)}(\J/k)$ is finite and effectively
computable. Computing $S^{(2)}(\J/k)$ gives an upper bound for the
rank, but this will only be sharp if the Tate-Shafarevich group
$\Sha(\J/k)$ contains no elements of order $2$.

In principle, the upper bound for the rank can be improved by carrying
out a $4$-descent, that is, by computing the $4$-Selmer group
$S^{(4)}(\J/k)$. This sits in a commutative diagram with exact rows
\begin{equation}
  \label{commdiag1}
 \begin{aligned}
 \xymatrix{ \J(k) \ar[r]^{\times 4} \ar[d]_{\times 2} & \J(k) \ar[r] \ar@{=}[d]
   & S^{(4)}(\J/k) \ar[d]_{\alpha} \ar[r] & \Sha(\J/k)[4]
   \ar[d]_{\times 2} \ar[r] & 0 \\
    \J(k) \ar[r]^{\times 2} & \J(k) \ar[r] & S^{(2)}(\J/k)
  \ar[r] & \Sha(\J/k)[2] \ar[r] & 0 } 
 \end{aligned}
\end{equation}
from which we read off the inclusions
$\J(k)/2 \J(k) \subset \im(\alpha) \subset S^{(2)}(\J/k)$.  However,
by results of Cassels \cite{CasselsIV} in the case of elliptic curves,
generalised to abelian varieties by Tate \cite{Tate}, \cite{Milne},
there is a symmetric pairing of $\F_2$-vector spaces
\[ \langle~,~\rangle_{\CT} : S^{(2)}(\J/k) \times S^{(2)}(\J/k) \to
  \F_2 \] whose kernel is $\im(\alpha)$. Thus a $2$-descent followed
by computing the Cassels-Tate pairing $\langle~,~\rangle_{\CT}$ gives
the same information (as regards bounding the rank) as a $4$-descent,
but with potentially very much less effort.

Cassels~\cite{Ca98} gave a practical method for computing the pairing
in the case of elliptic curves (that is, when $\C$ has genus $1$). His
method involves solving a conic over the field of definition of each
$2$-torsion point of the elliptic curve. This was subsequently
improved by Donnelly~\cite{Donnelly}, who found a method that only
requires that we solve conics over the ground field.

The second author \cite{Jiali-rat2tors} generalised the method of
Cassels to curves of genus~2. However the analogue of solving a conic
is now the problem of finding an explicit isomorphism between an
algebra given by structure constants and the algebra of $4 \times 4$
matrices.  So even taking $k = \Q$ the method is only practical if the
genus~$2$ curve has all its Weierstrass points defined over $\Q$, or
at least over a very small number field.  In this paper we develop a
new method that avoids these restrictions.  To do this we generalise
the method of Donnelly~\cite{Donnelly}, or more specifically the
variant in \cite{bq-ctp}, to curves of genus~$2$.

Although we are able to compute the pairing for curves of genus $2$
regardless of the Galois action on the Weierstrass points, our method
depends on being able to find rational points on certain twisted
Kummer surfaces. In general such points need not exist, as shown by
Logan and van Luijk~\cite{Logan-vanLuijk}. However this does not
appear to be a severe restriction in practice. Indeed we have
implemented our method in Magma \cite{magma}, and in running our
program on thousands of random examples, and on every genus 2 Jacobian
in the {\em L-functions and modular forms database} \cite{lmfdb} with
non-trivial analytic order of $\Sha$, we are yet to find an example
where a lack of rational points on the twisted Kummer surfaces
prevents us from computing the pairing.

There are also good theoretical reasons why the twisted Kummer
surfaces should typically have rational points. The surfaces of
interest are always everywhere locally soluble, and so have rational
points if they satisfy the Hasse principle. However Kummer surfaces
are examples of $K3$ surfaces, and for these it is conjectured that
all counterexamples to the Hasse principle are explained by the
Brauer-Manin obstruction.  Moreover there is work of Harpaz and
Skorobogatov~\cite{HS}, and more recently Morgan~\cite{Morgan},
showing that if we assume the finiteness of all relevant
Tate-Shafarevich groups, then twisted Kummer surfaces, satisfying
certain additional mild assumptions, do indeed satisfy the Hasse
principle.

Another approach to computing the Cassels-Tate pairing on the
$2$-Selmer group of a genus $2$ Jacobian (in fact, more generally, of
a hyperelliptic Jacobian) is currently being developed by Shukla and
Stoll. This is based on the Albanese-Albanese definition of the
Cassels-Tate pairing, due to Poonen and Stoll \cite{PS}. The
practicality of their method is not yet clear, but unlike our method,
it only applies to curves with a rational Weierstrass point.

In the case of elliptic curves, algorithms for $4$-descent have been
developed, but computing the Cassels-Tate pairing on the $2$-Selmer
group requires significantly less effort. For genus $2$ Jacobians the
contrast is even more stark as there is, to our knowledge, no
practical way of carrying out a $4$-descent.  Computing the
Cassels-Tate pairing on the $2$-Selmer group therefore gives us
information that could not otherwise be obtained.  In particular we
have used our methods to compute the ranks of some genus $2$
Jacobians, taken from the LMFDB and from an experiment of Bruin and
Stoll \cite{BruinStoll}, in cases where the rank had only previously
been computed conditional on the Birch Swinnerton-Dyer conjecture. In
fact prior to our work there were $69$ genus~$2$ curves in the LMFDB
that were unresolved in this sense, and we have resolved all of them.

This paper has its origins in Chapters 5 and 6 of the second author's
PhD thesis \cite{JialiThesis}. However in the intervening two years we
have succeeded in making the method significantly more practical.
Some of the main improvements concern explicit $2$-descent on genus
$2$ Jacobians, by which we mean the problem of converting elements of
the $2$-Selmer group, represented algebraically (as units modulo
squares in suitable \'etale algebras), to explicit equations for the
corresponding $2$-coverings. This problem was previously addressed by
Flynn, Testa and van~Luijk \cite{2-coverings}, extending earlier work
of Gordon and Grant \cite{GG}.  However we find it convenient to take
a different approach, based on the observation that, at least over the
complex numbers, the Jacobian of a genus 2 curve may be realised as
the variety of lines on the intersection of two quadrics in
$\PP^5$. This ``neoclassical approach'' is discussed in the book of
Cassels and Flynn \cite[Chapter 17]{CF}, and has its origins in the
geometric literature in works of Newstead~\cite{Newstead}, Narasimhan
and Ramanan~\cite{NR}, Reid~\cite{Reid} and Donagi~\cite{Donagi}.
More recently it has found important applications in arithmetic
statistics; see for example \cite{BG}, \cite{SW}.

The upshot for us is that by representing the $2$-Selmer group
elements as pairs of quadratic forms in $6$ variables, we obtain
simple elegant formulae for the $2$-coverings in $\PP^{15}$, the
twisted Kummer surfaces in $\PP^3$, the twisted desingularised Kummer
surfaces in $\PP^5$, the maps between these, and (various maps induced
by) the covering map to the Jacobian.  We have also found
invariant-theoretic formulae, analogous to those in \cite{bq-ctp},
that allow us to directly compute the $(2,2,2)$-form that appears in
our formula for the Cassels-Tate pairing.

We concentrate in this paper in giving practical methods for curves of
genus~$2$. The question as to how much of our work generalises to
hyperelliptic curves of higher genus is left to future work. One
reason why such a generalisation might not be automatic is that we
make implicit use of the exceptional isomorphism of Lie algebras
$\mathfrak{so}_6 \isom \mathfrak{sl}_4$.  The practicality of
searching for rational points on higher dimensional Kummer varieties
could also be a problem.

Throughout this paper $\C$ will be a genus $2$ curve, defined over a
field $k$ of characteristic not $2$, with equation
\begin{equation}
\label{eqn:C}  
y^2 = f(x) = f_6 x^6 + f_5 x^5 + \ldots + f_1 x + f_0.
\end{equation}
We assume that $f_6 \not= 0$, noting that if $\deg f = 5$ and $|k|>5$
then we can reduce to this case by applying a suitable M\"obius map.

We have divided the paper into three main sections, each with its own
introduction. Section~\ref{sec:2} gives the background on $2$-descent,
including the refinements mentioned above.  In Section~\ref{sec:CTP}
we develop our method for computing the Cassels-Tate pairing, and in
Section~\ref{sec:examples} we give examples and applications.

We plan to make some Magma code accompanying this article (checking
some of the formulae and examples) available from the first author's
website.

\medskip

\noindent {\bf{Acknowledgements.}} We thank Victor Flynn, Jef Laga,
August Liu, Adam Morgan, Ross Paterson, Himanshu Shukla, Michael
Stoll, Drew Sutherland and Jack Thorne for useful conversations.

\section{Explicit 2-descent on genus 2 Jacobians}
\label{sec:2}

We give an account of $2$-descent on the Jacobian of a genus $2$
curve.  The emphasis is on explicit formulae, some of which are taken
from the existing literature and some of which are new.

In Sections~\ref{sec:kum} and~\ref{sec:desing} we give formulae
relating to the Kummer surface, the dual Kummer surface, the
desingularised Kummer, the embedding of the Jacobian in $\PP^{15}$,
and the action of the $2$-torsion points.  We are particularly
interested in constructions that carry over to the twisted Kummer
surfaces considered in Section~\ref{sec:tw-kum}.  In
Section~\ref{sec:sel} we briefly recall the standard algebraic
description of the $2$-Selmer group, and define its canonical
element. Section~\ref{sec:models} introduces the ``neoclassical
approach'' where we represent Selmer group elements as pairs of
quadratics forms. We call these pairs of quadratic forms {\em
  models}. In Section~\ref{sec:recover} we describe some methods for
recovering a Selmer group element from a model. In
Section~\ref{sec:tw-kum} we explain how a model determines an equation
for the corresponding twisted Kummer surface, and a wealth of other
information.

\subsection{The Kummer surface and its dual}
\label{sec:kum}
Let $\C$ be a genus $2$ curve with equation $y^2 = f(x)$ where $f$ is
a polynomial of degree $6$, with coefficients labelled as
in~\eqref{eqn:C}. Let $\J$ be the Jacobian of $\C$, and let
$\K = \J/[- 1]$ be its Kummer surface. We may represent points on $\J$
as \[ [(x,y) + (x',y') - \kappa] \] where $(x,y),(x',y') \in \C$ and
$\kappa$ is the canonical divisor on $\C$. A point on $\K$ may
therefore by represented by $x, x'$ and a choice of square root of
$f(x)f(x')$. It follows that $\K$ is the double cover of the
projective plane $\PP^2_{a,b,c}$ with equation
\begin{equation}
\label{double_cover}
  z^2 = \Res(a x^2 - b x + c,f(x)).
\end{equation}
If we specialise $a,b,c$ to $1,2t,t^2$ then this resultant equals
$f(t)^2$. This suggests putting
\begin{align*}
  P_2(a,b,c) &= b^2 - 4ac, \\
  P_3(a,b,c) &= 2(2f_0 a^3 + f_1 a^2 b + 2f_2 a^2 c 
   + f_3 abc + 2f_4 ac^2 + f_5  b c^2 + 2 f_6 c^3),              
\end{align*}
so that $P_2(1,2t,t^2) = 0$ and $P_3(1,2t,t^2) = 4f(t)$.  We can then
solve for the homogeneous polynomial $P_4(a,b,c)$ of degree $4$
satisfying
\[ P_3^2 - 4 P_2 P_4 = 16 \Res(a x^2 - b x + c,f(x)). \]
The equation~\eqref{double_cover} for the Kummer surface becomes
\[ \K = \{ P_2 d^2 - P_3 d + P_4
  = 0 \} \subset \PP^3_{a,b,c,d}. \]

Renaming the coordinates $a,b,c,d$ as $x_1, \ldots, x_4$ we have the
following equation for $\K \subset \PP^3$, in agreement
with~\cite[page~19]{CF}.
\begin{equation}
  \label{kumeqn}
\begin{aligned} F & = 
   (x_2^2 - 4 x_1 x_3) x_4^2 - 2 (2 f_0 x_1^3 + f_1 x_1^2 x_2 +
  2 f_2 x_1^2 x_3 + f_3 x_1 x_2 x_3 + 2 f_4 x_1 x_3^2 \\ &
   + f_5 x_2 x_3^2 + 2 f_6 x_3^3) x_4
   + (f_1^2 - 4 f_0 f_2) x_1^4 - 4 f_0 f_3 x_1^3 x_2 -
  2 f_1 f_3 x_1^3 x_3 \\ & - 4 f_0 f_4 x_1^2 x_2^2
    + 4 (f_0 f_5 - f_1 f_4) x_1^2 x_2 x_3
    + (f_3^2 -4 f_0 f_6 + 2 f_1 f_5 - 4 f_2 f_4) x_1^2 x_3^2
    \\ & - 4 f_0 f_5 x_1 x_2^3 + 4 (2 f_0 f_6 - f_1 f_5) x_1 x_2^2 x_3
    + 4 (f_1 f_6 - f_2 f_5) x_1 x_2 x_3^2
   - 2 f_3 f_5 x_1 x_3^3 \\ & - 4 f_0 f_6 x_2^4
    - 4 f_1 f_6 x_2^3 x_3 - 4 f_2 f_6 x_2^2 x_3^2 - 4 f_3 f_6 x_2 x_3^3 +
    (f_5^2 - 4 f_4 f_6) x_3^4.
    \end{aligned}
\end{equation}
This quartic surface is also the image of $\J$ by the linear system
$|2 \Theta|$, where $\Theta$ is the theta divisor.\footnote{We abuse
  notation (suppressing the choice of Weierstrass point) as in
  \cite[page 397]{2-coverings}.}  It has exactly 16 nodes, these being
the images of the $2$-torsion points on $\J$.

The dual Kummer $\K^\vee \subset (\PP^3)^\vee$ is obtained by mapping
each smooth point on $\K \subset \PP^3$ to its tangent plane.  We
write $x^*_1, \ldots, x^*_4$ for the coordinates on $(\PP^3)^\vee$
dual to the coordinates $x_1, \ldots, x_4$ on $\PP^3$.  According to
\cite[page 33]{CF}, the dual Kummer has equation
\[ \left| \begin{array}{cccc}
  2 f_0 x^*_4 & f_1 x^*_4 & x^*_1 & x^*_2 \\
  f_1 x^*_4 & 2 f_2 x^*_4 - 2 x^*_1 & f_3 x^*_4 - x^*_2 & x^*_3 \\
  x^*_1 & f_3 x^*_4 - x^*_2 & 2 f_4 x^*_4 - 2 x^*_3 & f_5 x^*_4 \\
  x^*_2 & x^*_3 & f_5 x^*_4 & 2 f_6 x^*_4
\end{array} \right| = 0. \]

We identify $\Pic^d \C$ and $\Pic^{d+2}\C$ via addition of the
canonical divisor, so that $\Pic^d \C$ only depends on the parity of
$d$. It is a double cover of $\K$ if $d$ is even and a double cover of
$\K^\vee$ if $d$ is odd.  Adding a point $P \in \C$ defines a map
$\Pic^d \C \to \Pic^{d+1} \C$.  If $P = (\theta,0)$ is a Weierstrass
point then these maps descend to a pair of inverse maps
$\K \to \K^\vee$ and $\K^\vee \to \K$. According to \cite[pages 38 and
184]{CF}, the map $\K \to \K^\vee$ is given by
\begin{equation}
  \label{skew-symm}
\begin{pmatrix} x^*_1 \\ x^*_2 \\ x^*_3 \\ x^*_4 \end{pmatrix}
   = \begin{pmatrix}
    0 & h_5(\theta) & h_4(\theta) & \theta^2 \\
    -h_5(\theta) & 0 & h_3(\theta) & -\theta \\
    -h_4(\theta) & -h_3(\theta) & 0 & 1 \\
    -\theta^2 & \theta & -1 & 0
  \end{pmatrix}
  \begin{pmatrix} x_1 \\ x_2 \\ x_3 \\ x_4 \end{pmatrix}
  \end{equation}
  where
  \begin{equation}
  \label{defh}
\begin{aligned}
  h_3(\theta) &= 2 f_6 \theta^3 + f_5 \theta^2,  \\
  h_4(\theta) &= 2 f_6 \theta^4 + 2 f_5 \theta^3
                  + 2 f_4 \theta^2 + f_3 \theta, \\
  h_5(\theta) &= 2 f_6 \theta^5 + 2 f_5 \theta^4 + 2 f_4 \theta^3 + 2
  f_3 \theta^2 + 2 f_2 \theta + f_1.
\end{aligned}
\end{equation}

For $A = (A_{ij})$ a $4 \times 4$ skew symmetric matrix we
put\footnote{In basis free language, we are identifying the
  alternating square of a $4$-dimensional vector space with its dual.}
\begin{equation}
\label{def:star}
A^* = \begin{pmatrix}
  0 & A_{43} & A_{24} & A_{32} \\ 
  A_{34} & 0 & A_{41} & A_{13} \\ 
  A_{42} & A_{14} & 0 & A_{21} \\ 
  A_{23} & A_{31} & A_{12} & 0 \\
\end{pmatrix}.
\end{equation}
This is again a skew symmetric matrix. The Pfaffian $\Pf(A)$ then
satisfies
\begin{equation}
\label{star-adj}
A A^*  = A^* A = \Pf(A) I_4.
\end{equation}
We note that $\Pf(A)^2 = \det(A)$, $\Pf(A^*) = \Pf(A)$ and
$A^{**} = A$.  Moreover for any $4 \times 4$ matrix $P$ we have
\begin{equation}
  \label{pf-det}
\Pf(P^T A P) = \det(P) \Pf(A).
\end{equation}  

Writing $A$ for the matrix in \eqref{skew-symm}, we have
\begin{equation}
\label{Pf0}
 \Pf(A) = h_5(\theta) + \theta h_4(\theta) + \theta^2 h_3(\theta)
 = f'(\theta).
\end{equation} 
Let $f$ have roots $\theta_1, \ldots, \theta_6$, and write $A_i$ for
the matrix obtained by taking $\theta = \theta_i$ in
\eqref{skew-symm}.  As explained in \cite[page 38]{CF}, for
$i \not= j$ the action of
$T_{ij} = [(\theta_i,0) + (\theta_j,0) - \kappa] \in \J[2]$ by
translation on $\K \subset \PP^3$ is given by
\begin{equation}
\label{MT}
A_i^* A_j = - A_j^* A_i.
\end{equation}
Since $T_{12} + T_{34} + T_{56} = 0$ it follows that
$A_1^* A_2 A_3^* A_4 A_5^* A_6 = \lambda I_4$ for some scalar
$\lambda$.  Squaring both sides and using~\eqref{star-adj},
\eqref{Pf0} and \eqref{MT} shows that
$\lambda^2 = f_6^6 \prod_{i<j} (\theta_i-\theta_j)^2$.  Making the
correct choice of square root it turns out that
\begin{equation}
\label{eqn:prod0}
A_1^* A_2 A_3^* A_4 A_5^* A_6 = f_6^3 \prod_{i<j} (\theta_i -\theta_j) I_4. 
\end{equation}
An alternative, less symmetric, way to write this identity is as
\begin{equation}
   \label{3-3}
   \frac{ A_1 A_2^* A_3 }{ (\theta_1 - \theta_2)
     (\theta_2 - \theta_3) (\theta_3 - \theta_1) }
   =  \frac{ A_4 A_5^* A_6 } {  (\theta_4 - \theta_5)
      (\theta_5 - \theta_6) (\theta_6 - \theta_4) }.
\end{equation}
\begin{Remark}
\label{rem:6-10}
The skew symmetric matrices $A_i$ define $6$ linear isomorphisms
$\K \to \K^\vee$. A further $10$ such isomorphisms are defined by the
symmetric matrices of the form~\eqref{3-3}. These matrices are indexed
by the $(1,5)$-partitions and $(3,3)$-partitions of the Weierstrass
points.
\end{Remark}

\subsection{The desingularised Kummer}
\label{sec:desing}
  
Blowing up the $16$ nodes of $\K$ (or equally of $\K^\vee$) gives a
surface $\dKum$ called the desingularised Kummer. Equations for
$\dKum$ may be obtained as follows; see for example \cite[Section
4]{2-coverings}. Let $L=k[x]/(f)=k[\theta]$ and define quadratic forms
$Q_0,Q_1, \ldots, Q_5 \in k[u_0,u_1, \ldots, u_5]$ by
\begin{equation}
\label{desing1}
    (u_0 + u_1 \theta + \ldots + u_5 \theta^5)^2 = Q_0 + Q_1 \theta +
    \ldots + Q_5 \theta^5.
\end{equation}
Then $\dKum = \{Q_3 = Q_4 = Q_5 = 0\} \subset \PP^5$.  The quadratic
form $Q_5$ only involves the monomials $u_i u_j$ with
$i + j \geqslant 5$, and so has a 3-dimensional isotropic subspace
given by $u_3 = u_4 = u_5 = 0$. This suggest making a change of
coordinate on $\PP^5$ transforming $Q_5$ to a scalar multiple of
$u_0 u_5 + u_1 u_4 + u_2 u_3$.  Such a change of coordinates is given
by replacing the basis $1, \theta, \ldots, \theta^5$ for $k[\theta]$
on the left hand side of~\eqref{desing1} by the basis
\[ 1, \theta, \theta^2, h_3(\theta), h_4(\theta), h_5(\theta). \]
where the $h_j(\theta)$ were defined in~\eqref{defh}.  Making the same
change of basis on the right hand side of~\eqref{desing1}, we find
that $Q_5 = 2G$, $Q_4 = 2H$ and $Q_3 = 2 S$ where
 \begin{align*}
   G & = u_0 u_5 + u_1 u_4 + u_2 u_3, \\
   H & = u_0 u_4 + u_1 u_3 - f_0 u_5^2 - f_1 u_4 u_5 - f_2 u_4^2
       - f_3 u_3 u_4 - f_4 u_3^2 + \tfrac{1}{4}  f_6^{-1}
       (u_2 - f_5 u_3)^2, \\
   S &= u_0 u_3 - 2 f_0 u_4 u_5 - f_1 (u_3 u_5 + u_4^2)
       - 2 f_2 u_3 u_4 - f_3 u_3^2 \\ & \qquad \qquad
       + \tfrac{1}{2} f_6^{-1}  (u_1 - f_3 u_4 - 2 f_4 u_3) (u_2 - f_5 u_3)
       - \tfrac{1}{4} f_6^{-2} f_5 (u_2 - f_5 u_3)^2.
\end{align*}
The desingularised Kummer is now given by
$\dKum = \{G = H = S = 0\} \subset \PP^5$.

Writing $\bG$, $\bH$, $\bS$ for the symmetric matrices with
$G(u) = \frac{1}{2} \bu^T \bG \bu$ etc., where
$\bu = (u_0, \ldots, u_5)^T$, we find that $\bS = \bH \bG^{-1} \bH$
and
\begin{equation*}
\label{det1}
   \det (x \bG - \bH) = - f_6^{-1} f(x).
\end{equation*}
In particular the $6 \times 6$ matrix $\theta \bG - \bH$ is
singular. In fact this matrix has rank~5 and its kernel is spanned by
the vector with entries
\begin{equation}
\label{Phi0}
 h_5(\theta), h_4(\theta), h_3(\theta), \theta^2, \theta, 1.
\end{equation}

We may associate to each point $(x_1:x_2:x_3:x_4) \in \PP^3$ a
$3$-dimensional isotropic subspace for $G$ spanned by the rows of the
matrix
\begin{equation}
\label{def:Lambda}
   \Lambda = \begin{pmatrix}
      0 & 0 & x_4 & 0 & x_3 & x_2 \\
      0 & -x_4 & 0 & x_3 & 0 & -x_1 \\
      x_4 & 0 & 0 & -x_2 & -x_1 & 0 \\
      -x_3 & x_2 & -x_1 & 0 & 0 & 0 
\end{pmatrix}.
\end{equation}
The ``neoclassical approach'' 
(see the introduction for references) suggests we look for lines
contained in $\{ G = H = 0 \} \subset \PP^5$.  We obtain two such
lines if the restriction of $H$ to a $3$-dimensional isotropic
subspace for $G$ has rank $2$. This suggests that $\K \subset \PP^3$
should be defined by the $3 \times 3$ minors of
$B = \Lambda \bH \Lambda^T$. This is confirmed by a direct
calculation, which more precisely shows that the equation $F$ for $\K$
in~\eqref{kumeqn} satisfies
\begin{equation}
\label{getF}
 -2f_6 \adj B = F (x_1, \ldots, x_4) \x \x^T
\end{equation}
where $\x = (x_1,x_2,x_3,x_4)^T$.

The equations for $\dKum$ may be written in terms of Pfaffians. To do
this we define skew symmetric matrices of linear forms in
$u_0, \ldots, u_5$ by
\begin{equation}
\label{def:ZW}
Z = \begin{pmatrix}
   0 & u_0 & u_1 & u_3 \\
   & 0 & u_2 & -u_4 \\
   &  & 0 & u_5 \\
   & & & 0
\end{pmatrix} \quad \text{ and } \quad
W = \begin{pmatrix}
   0 & H_0 & H_1 & H_3 \\
   & 0 & H_2 & -H_4 \\
   &  & 0 & H_5 \\
   & & & 0
\end{pmatrix}^*
\end{equation}
where $H_i = \partial H/\partial u_i$, and the superscript $*$ has the
meaning explained in~\eqref{def:star}. Then for indeterminates
$\lambda$ and $\mu$ we have
 \[ \Pf( \lambda Z + \mu W) = \lambda^2 G - 2 \lambda \mu H + \mu^2 S. \]

\begin{Remark}
\label{rem:Kmaps}
The birational map between $\K \subset \PP^3$ and
$\dKum \subset \PP^5$ may be described in the following ways.
\begin{enumerate}
\item The graph of this map in $\PP^3 \times \PP^5$ is defined by the
  8 bilinear forms
\begin{equation*}
(x_1 \,\, x_2 \,\, x_3 \,\, x_4) (Z | W ) = 0.
\end{equation*}
\item The $4 \times 4$ matrix of quadratic forms $Z W^*$ has rank at
  most $1$ on $\dKum$.  Any row defines the map $\dKum \to \K$ and any
  column defines the map $\dKum \to \K^\vee$.
\item The $2 \times 2$ minors of $B$ may be arranged in a $6 \times 6$
  symmetric matrix of quartic forms that has rank at most $1$ on $\K$.
  Any row or column defines the map $\K \to \dKum$.  Setting any
  diagonal entry of the $6 \times 6$ matrix equal to $-f_6$ times a
  square defines the double cover $\J \to \K$.
\end{enumerate}
\end{Remark}

\begin{Remark}
  The previous remark is closely related to the 72 quadratic equations
  defining $\J \subset \PP^{15}$, as originally computed by
  Flynn~\cite{72}, and revisited in \cite{2-coverings}. We write our
  coordinates on $\PP^{15}$ as $x_{ij}$ and $z_{ij}$ where these are
  the entries of a generic $4 \times 4$ symmetric matrix $X$, and a
  generic $4 \times 4$ skew symmetric matrix $Z$. We identify the
  $z_{ij}$ with the $u_i$ via \eqref{def:ZW}. Our first $21$ quadratic
  equations are the $2 \times 2$ minors of $X$.  Next the entries of
  $XZ$ and $XW$ give $32$ quadratic equations, in fact only spanning a
  space of dimension $30$, since these matrices have trace zero.
  Finally, generalising the last part of Remark~\ref{rem:Kmaps}(iii),
  there are $21$ quadratic equations
  $-f_6 z_{ij} z_{kl} = B_{ik} B_{jl} - B_{il} B_{jk}$ where each
  quartic on the right is rewritten as a quadratic using
  $x_{rs} = x_r x_s$. In view of the first $21$ equations, the choices
  here do not matter. In total this gives $21 + 30 + 21 = 72$
  equations.
\end{Remark}

\subsection{The Selmer and fake Selmer groups}
\label{sec:sel}

We continue to take $\C$ a genus $2$ curve with equation $y^2 = f(x)$
where $f \in k[x]$ is a polynomial of degree $6$. We now suppose that
$k$ is a number field. Let $\J$ be the Jacobian of $\C$.  The
$2$-Selmer group is
\begin{equation}
\label{def:sel}
S^{(2)}(\J/k) = \ker\left( H^1(k,\J[2]) \to \prod_v H^1(k_v,\J) \right).
\end{equation}

Let $W$ be the set of Weierstrass points on $\C$, that is, the points
$(\theta,0)$ where $\theta$ is a root of $f$.  Let $\Phi$ be the set
of all ways of partitioning $W$ into two subsets. There is a natural
Galois action on $\Phi$ induced by the Galois action on $W$, and a
group law given by symmetric difference, equivalently by pointwise
addition of indicator functions in $\Map(W,\F_2)/\F_2$. We write
$\Phi = \Phi_0 \sqcup \Phi_1$ where $\Phi_0$ (resp. $\Phi_1$) is the
set of decompositions into subsets of even (resp. odd) size.  It is
well known that $\Phi_0$ is isomorphic to $\J[2]$ as a Galois
module. It follows that $\Phi_1$ is a torsor under $\J[2]$.  We write
$c \in H^1(k,\J[2])$ for the class of $\Phi_1$.

\begin{Lemma}
\label{lemL}
The canonical element $c$ belongs to $S^{(2)}(\J/k)$.
\end{Lemma}
\begin{proof}
  The canonical element $c$ is represented by the cocycle
  $\sigma \mapsto [\sigma P - P]$ where $P \in \C$ is a Weierstrass
  point. It follows that the natural map $H^1(k,\J[2]) \to H^1(k,\J)$
  sends $c$ to the class of $\Pic^1 \C$. The lemma is therefore
  equivalent to showing that $\C$ admits a $k_v$-rational divisor
  class of degree $1$ for each place $v$ of~$k$. As noted by Poonen
  and Stoll~\cite[Lemma~1]{PS} this is a consequence of Tate local
  duality, and is originally due to Lichtenbaum~\cite[Theorem
  7]{Lichtenbaum-dualitythm}.
\end{proof}  

\begin{Definition}
  The fake Selmer group is
  $S_{\rm fake}^{(2)}(\J/k) = S^{(2)}(\J/k)/\langle c \rangle$.
\end{Definition}

The analogue of~\eqref{def:sel} is
\[ S_{\rm fake}^{(2)}(\J/k) = \ker \left( \frac{H^1(k,\J[2])}{\langle c
      \rangle} \to \prod_v H^1(k_v,\J)  \right). \]

Let $L = k[x]/(f) = k[\theta]$ be the \'etale algebra of $W$. As shown
by Poonen and Schaefer \cite{PoonenSchaefer} there is an isomorphism
\begin{equation}
\label{messy-isom}
 \frac{H^1(k,\J[2])^{\cup c = 0}}{\langle c \rangle}
 \isom \ker \left( L^\times/(k^\times (L^\times)^2)
    \stackrel{N_{L/k}}{\ra} k^\times/(k^\times)^2 \right)
\end{equation}
where the superscript ${\cup c = 0}$ indicates the subgroup of
elements that pair trivially with $c$ under the pairing
\[H^1(k,\J[2]) \times H^1(k,\J[2]) \to \Br(k)\] given by cup product
and the Weil pairing. The composite of the connecting map
$\J(k)/2\J(k) \to H^1(k,\J[2])$ and~\eqref{messy-isom} is the Cassels
map (see e.g.~\cite{Cassels-G2}, \cite{FPS}), given for $yy' \not=0$ by
\begin{equation}
\label{cassels-map}
[(x,y)+(x',y')-\kappa] \mapsto (x - \theta)(x' - \theta).
\end{equation}
The fake Selmer group $S_{\rm fake}^{(2)}(\J/k)$ naturally arises as
the subgroup of the right hand side of~\eqref{messy-isom} consisting
of elements that are everywhere locally in the image of the Cassels
map.  This group can be computed in Magma \cite{magma}, using programs
originally due to Stoll~\cite{Stoll}.

It is easy to compute the dimension of $S^{(2)}(\J/k)$ as an
$\F_2$-vector space from that of $S_{\rm fake}^{(2)}(\J/k)$, simply by
adding $0$ or $1$ according as the canonical element $c$ is trivial or
non-trivial.  However we shall need a more explicit way of
representing elements of $S^{(2)}(\J/k)$. As explained by Stoll and
van Luijk \cite{unfaking}, this is given by augmenting elements of the
right hand side of~\eqref{messy-isom} with a choice of square root of
the norm. Thus $S^{(2)}(\J/k)$ is a subgroup of
\begin{equation}
\label{xi-m-pairs}
 H^1(k,\J[2])^{\cup c = 0} \isom \frac{\{ (\xi,m)
    \in L^\times \times k^\times
    | N_{L/k}(\xi) = m^2 \}}{ \{(r \nu^2, r^3 N_{L/k}(\nu)) |
    r \in k^\times, \nu \in L^\times \}}.
\end{equation}

\subsection{Models for 2-coverings} 
\label{sec:models}
Following the ``neoclassical approach'' (see the introduction for
references) we explain how elements of the $2$-Selmer group
$S^{(2)}(\J/k)$ may be represented by certain pairs of quadratic
forms, which we call {\em models}.

We only need the following proposition in the case $n=6$, but the
extra generality costs us nothing.
\begin{Proposition}
  \label{lem:viadiff}
  Let $f(x) = f_n x^n 
  + \ldots + f_1 x + f_0 \in k[x]$ be a square-free polynomial of
  degree $n$.  Let $L = k[\theta] = k[x]/(f)$ and fix
  $\xi \in L^\times$. Consider the quadratic forms
  $Q_j \in k[u_0, \ldots,u_{n-1}]$ defined by
  \begin{equation}
    \label{def:Q}
   \xi(u_0 + u_1 \theta + \ldots + u_{n-1} \theta^{n-1})^2
   = \sum_{j=0}^{n-1} Q_j (u_0, \ldots, u_{n-1}) \theta^j.
   \end{equation}
    \begin{enumerate}
    \item There are $n \times n$ symmetric matrices $\bG$ and $\bH$,
      with $\bG$ invertible, such that
      \begin{equation*}
        \qquad Q_j(u_0, \ldots, u_{n-1})
        = f_{n}^{-1} \sum_{i = j+1}^n f_i \bu^T \bG
        (\bG^{-1} \bH )^{i-j-1} \bu  \end{equation*}
      for all $0 \leqslant j \leqslant n - 1$,
      where $\bu = (u_0, \ldots, u_{n-1})^T$.
    \item We have \[\det(x \bG - \bH) =
        (-1)^{\binom{n}{2}} f_n^{-1} N_{L/k}(\xi) f(x).\]
\item If $\vv = \bG^{-1} (1 ,\theta, \ldots, \theta^{n-1})^T$ then
  $(\theta \bG - \bH) \vv = 0$ and $\vv^T \bG \vv = f_n^{-1} f'(\theta)/\xi$.
\item We have
  \[ \qquad \left( \frac{1}{2^n} \, \frac{\partial(Q_0, \ldots,
        Q_{n-1})}{\partial(u_0,\ldots,u_{n-1})} \right)^2 =
    (-1)^{\binom{n}{2}} (\det \bG) N_{L/k} \left( \sum_{j=0}^{n-1} Q_j
      (u_0, \ldots, u_{n-1}) \theta^j \right). \]
 \end{enumerate}
\end{Proposition}
\begin{proof}
  Let $\theta_1, \ldots, \theta_n$ be the roots of $f$. A well known
  identity that is used in the theory of the different (see for
  example \cite[page 45]{brief}) states that
  \[ \sum_{r=1}^n \frac{f(x)}{x - \theta_r} \,
    \frac{\theta_r^i}{f'(\theta_r)} = x^i. \] for all
  $0 \leqslant i \leqslant n-1$.  The proof is simply that each side
  is a polynomial of degree less than $n$ taking the same values at
  $\theta_1, \ldots, \theta_n$.  Writing
\begin{equation}
\label{def:betas}
f(x) = (x- \theta)(\beta_{n-1} x^{n-1} + \ldots + \beta_1 x + \beta_0)
\end{equation}
it follows that the $\theta^i$ and $\beta_i/f'(\theta)$ are dual bases
for $L$ with respect to the trace pairing
$(x,y) \mapsto \tr_{L/k}(xy)$.  In particular for any $a \in L$ we
have
\begin{equation}
\label{eqn:trace2}
N_{L/k}(a) = \det \left( \tr_{L/k} \left( \frac{a \theta^i \beta_j}{f'(\theta)}
  \right)_{i,j = 0, \ldots,n-1} \right)
\end{equation}
and~\eqref{def:Q} is satisfied with
\begin{equation}
  \label{def:Qj}
 Q_j(u_0, \ldots, u_{n-1}) = \tr_{L/k} \left( \frac{ \xi \beta_j
     (u_0 + \ldots + u_{n-1} \theta^{n-1})^2}{f'(\theta)} \right).
 \end{equation}

\noindent
(i) Let $\bG$ represent the pairing
$(x,y) \mapsto \tr_{L/k}( f_n \xi xy/f'(\theta))$, and $\bTheta$
represent the multiplication-by-$\theta$ map, both with respect to the
basis $1, \theta, \ldots, \theta^{n-1}$.  Comparing coefficients
in~\eqref{def:betas} shows that
$\beta_j = \sum_{i=j+1}^n f_i \theta^{i-j-1}$.
Rewriting~\eqref{def:Qj} using this notation and putting
$\bH = \bG \bTheta$ gives the stated formula.

\noindent
(ii) We have $\det(x \bG - \bH) = \det \bG \det(x I_n - \bTheta)$
where
\begin{equation}
\label{detG}
\det \bG = N_{L/k} (\xi)  \Delta(1,\theta, \ldots,
\theta^{n-1}) /N_{L/k}(f_n^{-1} f'(\theta)) = (-1)^{\binom{n}{2}} N_{L/k} (\xi)
\end{equation}
and $\det(x I_n - \bTheta) = f_n^{-1} f(x)$.

\noindent
(iii) We have $\vv^T (\theta \bG - \bH) =
(1, \theta, \ldots, \theta^{n-1}) (\theta I_n - \bTheta) = 0$.
On the other hand $\bG \vv = (1, \theta, \ldots, \theta^{n-1})^T$
and $\vv = (f_n \xi)^{-1} (\beta_0,\beta_1, \ldots, \beta_{n-1})^T$. 
Therefore \[\vv^T \bG \vv = (\beta_0 + \beta_1 \theta +
\ldots + \beta_{n-1} \theta^{n-1} )/(f_n \xi) = f'(\theta)/(f_n \xi). \]

\noindent
(iv) We differentiate~\eqref{def:Qj} to get
\[ \frac{\partial Q_j}{\partial u_i}(u_0, \ldots, u_{n-1}) = \tr_{L/k}
  \left( \frac{ 2 \xi \theta^i \beta_j (u_0 + \ldots + u_{n-1}
      \theta^{n-1})}{f'(\theta)} \right). \]
It follows by~\eqref{eqn:trace2} that
\[ \frac{\partial(Q_0, \ldots, Q_{n-1})}{\partial(u_0,\ldots,u_{n-1})} 
  = N_{L/k} \left( 2 \xi (u_0 + u_1 \theta +  \ldots + u_{n-1} \theta^{n-1})
   \right). \]
Squaring both sides and then using~\eqref{def:Q} and~\eqref{detG}
gives the stated formula.
\end{proof}

We now return to considering a genus $2$ curve $\C$ with equation
$y^2 = f(x)$ where $f \in k[x]$ is a polynomial of degree $6$.  Let
$\J$ be the Jacobian of $\C$.

\begin{Definition}
  \label{def:model}
  (i) A {\em model} ({\em for a $2$-covering of $\J$}) is a pair of
  quadratic forms $(G,H)$ in variables $u_0, \ldots, u_5$, where
  $G = u_0 u_5 + u_1 u_4 + u_2 u_3$ and the corresponding symmetric
  matrices (given by $G(\bu) = \frac{1}{2} \bu^T \bG \bu$ and
  $H(\bu) = \frac{1}{2} \bu^T \bH \bu$) satisfy
\begin{equation}
  \label{det:cond}
  \det(x \bG - \bH) = -f_6^{-1} f(x).
\end{equation}
Since $G$ is fixed,
we sometimes refer to the model $(G,H)$ simply as $H$. \\
(ii) Two models are {\em $k$-equivalent} if they are in the same orbit
for the action of $\PGL_4(k)$ defined as follows. We identify the
$4 \times 4$ skew symmetric matrices $Z$ with the column vectors
$\bu = (u_0, \ldots, u_5)^T$ via~\eqref{def:ZW}, so that $G =
\Pf(Z)$. Then the action of $P \in \GL_4$ on the space of such
matrices via $Z \mapsto P Z P^T$ is described by a matrix
$\wedge^2 P \in \GL_6$.  We let $P \in \GL_4$ act on a model $(G,H)$
first by changing coordinates using $\wedge^2 P$ and then dividing
both $G$ and $H$ by $\det P$. By~\eqref{pf-det} this preserves $G$.
Since the scalar matrices act trivially, this is an action of
$\PGL_4$.
\end{Definition}
  
\begin{Remark}
  (i) If we change our equation $y^2 = f(x)$ for $\C$ by applying a
  M\"obius map to the $x$-coordinate then, provided $f$ still has degree
  $6$, we may update our models by applying the inverse M\"obius map to
  the transformation $\bTheta = \bG^{-1} \bH$ in the proof
  of Proposition~\ref{lem:viadiff}. \\
  (ii) Any model for a $2$-covering of $\J = \Jac(\C)$ is also a model
  for a $2$-covering of $\J_d = \Jac(\C_d)$ where
  $\C_d : d y^2 = f(x)$ is any quadratic twist of $\C$.
\end{Remark}

Starting from $(\xi,m) \in S^{(2)}(\J/k)$ we compute a model as
follows. First, analogous to~\eqref{desing1}, we define quadratic
forms $Q_0,Q_1, \ldots, Q_5 \in k[u_0,u_1, \ldots, u_5]$ by
\begin{equation}
\label{desing2}
   \xi (u_0 + u_1 \theta + \ldots + u_5 \theta^5)^2 = Q_0 + Q_1 \theta +
    \ldots + Q_5 \theta^5.
\end{equation}
If $\xi$ is in the image of the Cassels map~\eqref{cassels-map} then,
after multiplying by a square, it is quadratic in $\theta$.  This
motivates our interest in rational points on the twisted
desingularised Kummer surface $\{Q_3 = Q_4 = Q_5 = 0\} \subset \PP^5$.
As suggested by Proposition~\ref{lem:viadiff}(i) with $n=6$ (see also
\cite[page 167]{CF} or \cite[page 403]{2-coverings}), we rewrite our
equations for this surface as
\begin{align*}
    G &= Q_5/2, \\
    H &= (f_6 Q_4 - f_5 Q_5)/(2 f_6), \\
    S &= (f_6^2 Q_3 - f_5 f_6 Q_4 + (f_5^2 - f_4 f_6) Q_5)/(2 f_6^2),
\end{align*}
The corresponding symmetric matrices (with
$G(\bu) = \frac{1}{2} \bu^T \bG \bu$ etc.) are then related by
$\bS = \bH \bG^{-1} \bH$.  By Proposition~\ref{lem:viadiff}(ii) we
have
\begin{equation}
\label{det-GH}
\det( x \bG - \bH) = - f_6^{-1} N_{L/k} (\xi) f(x).
\end{equation}

If $\xi$ is in the image of the Cassels map~\eqref{cassels-map} then,
after multiplying by a square, it is quadratic in $\theta$.  If we
further take $u_2 = u_3 = u_4 = u_5 = 0$ then the left hand side
of~\eqref{desing2} has no term $\theta^5$.  It follows that the
quadratic form $G$ has a $2$-dimensional isotropic subspace, and is
therefore the orthogonal direct sum of two hyperbolic planes and a
binary quadratic form. As a special case of~\eqref{det-GH} we have
$\det \bG = -N_{L/k}(\xi)$.  Since $N_{L/k}(\xi) = m^2$ the binary
quadratic form has discriminant a square. Therefore $G$ has a
$3$-dimensional isotropic subspace.\footnote{Our argument is specific
  to curves of genus $2$. However corresponding results are known for
  hyperelliptic curves in general; see \cite{Thorne}, \cite{Wang}.}

We are interested in $(\xi,m)$ representing an element of the Selmer
group. This means that $\xi$ is everywhere locally in the image of the
Cassels map.  The previous paragraph shows that $G$ has a
$3$-dimensional isotropic subspace everywhere locally. By a suitable
form on the the Hasse principle (for example, the weak Hasse principle
in \cite[Chapter 6]{Cassels-RCF}) it follows that $G$ has a
$3$-dimensional isotropic subspace globally.  In the case $k = \Q$ we
use the existing function {\tt IsotropicSubspace} in Magma
\cite{magma} to find such a subspace.  It is then easy to find a
change of coordinates putting $G$ in the standard form specified in
Definition~\ref{def:model}(i).  We can and do choose this change of
coordinates so that it has determinant $m$ (rather than $-m$).  This
gives the desired model $(G,H)$.

It is shown in \cite[Lemma 17.1.1]{CF} that the $3$-dimensional
isotropic subspaces of $G$ fall into two algebraic families.  This may
conveniently be seen as follows. We write $G = \Pf(Z)$ as in
Definition~\ref{def:model}(ii).  Setting all entries in the $j$th row
and column to zero makes $G$ vanish, as does setting all entries
outside the $j$th row and column to zero. Repeating these observations
for linear combinations of the rows and columns we see that the
$3$-dimensional isotropic subspaces of $G$ are parametrised by
$\PP^3 \sqcup (\PP^3)^\vee$.

Following \cite[Section 19.1]{FH}, the subgroup
$\PO(G) \subset \PGL_6$ preserving $G$ up to scalars sits in an exact
sequence
\[ 0 \to \PSO(G) \to \PO(G) \to \{ \pm 1\} \to 0 \] where
$\PGL_4 \isom \PSO(G)$ via $P \mapsto \wedge^2 P$, and $\{ \pm 1\}$
swaps over the two families of isotropic subspaces.  It follows that
our procedure (noting in particular the condition on the determinant
of the change of coordinates) associates to each Selmer group element
$(\xi,m)$ a unique equivalence class of models.

\begin{Remark}
  \label{rem:obs}
  If we started with an arbitrary pair $(\xi,m)$
  in~\eqref{xi-m-pairs}, not necessarily a Selmer group element, then
  the $3$-dimensional isotropic subspaces of $G$ would be parametrised
  by $S \sqcup S^\vee$, where $S$ is a $3$-dimensional Brauer-Severi
  variety, and $S^\vee$ is its dual. The obstruction map
  $H^1(k,\J[2]) \to \Br(k)$, as defined in \cite[Section~5]{ob}, is
  then realised by $(\xi,m) \mapsto [S]$.
\end{Remark}

\subsection{Recovering the Selmer group element}
\label{sec:recover}
We explain how to recover a Selmer group element $(\xi,m)$ from a
model $(G,H)$. This not only helps us check that the models we
computed in the last section are correct, but is also useful for
testing equivalence of models, and for realising the group law.

As in Section~\ref{sec:desing} the matrix $\theta \bG - \bH$ has
rank~$5$.  We pick a vector spanning its kernel, and rewrite it
using~\eqref{def:ZW} as a $4 \times 4$ skew symmetric matrix $A$ with
entries in $L = k[x]/(f) = k[\theta]$. We claim that
\begin{equation}
  \label{xi-Pf}
  \xi \equiv \Pf(A)/f'(\theta) \mod{k^\times (L^\times)^2}.
\end{equation}
In the untwisted case (i.e., when $\xi=1$) this claim follows
by~\eqref{skew-symm}, \eqref{Pf0} and~\eqref{Phi0}.  It is also easy
to see that the class~\eqref{xi-Pf} only depends on the equivalence
class of $(G,H)$. Indeed picking a different basis for the kernel of
$\theta \bG - \bH$ has the effect of multiplying $A$ by an element of
$L^\times$ and hence $\Pf(A)$ by an element of $(L^\times)^2$.
By~\eqref{pf-det} the only further effect of replacing $(G,H)$ by an
equivalent model is to multiply $\Pf(A)$ by an element of $k^\times$.

If we take $\xi = \Pf(A)/f'(\theta)$ then the matrix $A$ also
determines a square root $m$ of $N_{L/k}(\xi)$ via the formula
\begin{equation}
\label{getm}
A_1^* A_2 A^*_3 A_4 A_5^* A_6
= m f_6^3 \prod_{i<j} (\theta_i-\theta_j) I_4,
\end{equation}
which generalises~\eqref{eqn:prod0}.

Proposition~\ref{lem:viadiff}(iii) shows that the class of $\xi$ in
$L^\times/k^\times (L^\times)^2$ is given by evaluating the quadratic
form $G$ at a vector in the kernel of $\theta \bG - \bH$, and then
dividing by $f'(\theta)$. Since in computing a model we change
coordinates so that $G$ becomes the Pfaffian, this
proves~\eqref{xi-Pf}. Following this proof gives another formula for
$m$, which improves on~\eqref{getm} in that it avoids the need for
field extensions.  Indeed we may take $m = (\det N)/(2 f_6)^3$ where
$N$ is the $6 \times 6$ matrix that writes the entries of $A$ above
the diagonal in terms of the basis $1, \theta, \ldots, \theta^5$ for
$L$.

\begin{Remark}
\label{addc}
If we change $(G,H)$ by reversing the coordinates $u_0, \ldots, u_5$
then $A$ is replaced by $A^*$. Since $\Pf(A) = \Pf(A^*)$ this does not
change $\xi$, but using~\eqref{MT} and~\eqref{getm}, or the formula in
the last paragraph, it does change the sign of $m$.  This in turn
corresponds to adding the canonical element $c$ (as defined in
Section~\ref{sec:sel}).
\end{Remark}

The following alternative way to recover a Selmer group element
$(\xi,m)$ from a model $(G,H)$ will be useful in
Sections~\ref{sec:tw-kum} and~\ref{sec:get222}.

\begin{Lemma}
\label{lem:XiM}
Let $(G,H)$ be a model and let $Q_0, \ldots, Q_5$ be the quadratic
forms defined by the formula in Proposition~\ref{lem:viadiff}(i).
Then the forms
\[ \Xi(u_0, \ldots, u_5) = \sum_{j=0}^5 Q_j( u_0, \ldots, u_5)
  \theta^j \,\,\,\, \text{ and } \,\,\,\, M(u_0, \ldots, u_5) =
  \frac{1}{2^6} \, \frac{\partial(Q_0, \ldots,
    Q_{5})}{\partial(u_0,\ldots,u_{5})} \] satisfy
$N_{L/k}(\Xi) = M^2$. Moreover specialising these forms at any
$a \in k^6$ with $M(a) \not= 0$ gives a pair $(\xi, m)$ representing
the equivalence class of the model.
\end{Lemma}
\begin{proof}
  The first statement is proved by checking that
  Proposition~\ref{lem:viadiff}(iv) is invariant under all changes of
  coordinates, and then noting that for $G$ in the standard form (as
  specified in Definition~\ref{def:model}(i)) we have $\det \bG =
  -1$. The second statement follows by~\eqref{def:Q}.
\end{proof}

\subsection{Twisted Kummer surfaces}
\label{sec:tw-kum}

We may also represent elements of the $2$-Selmer group as
$2$-coverings.  In fact any $\eps \in H^1(k,\J[2])$ may be represented
by a $2$-covering $(\J_\eps,\pi_\eps)$. By definition this fits in a
commutative diagram
\begin{equation}  
\label{2cov}
\begin{aligned}  
  \xymatrix{ \J_\eps \ar[d]_{\phi_\eps} \ar[dr]^{\pi_\eps} \\
    \J \ar[r]_{\times 2} & \J }
\end{aligned}
\end{equation}
where $\phi_\eps$ is an isomorphism defined over $\kbar$.  We
emphasise that $\phi_\eps$ is not part of the data defining the
$2$-covering. Indeed we are free to change it by composing with
translation by $T$ for any $T \in \J[2]$.

We give $\J_\eps$ the structure of principal homogeneous space under
$\J$ via
\begin{equation}
  \label{phs}
  \begin{aligned} \J \times \J_\eps & \to \J_\eps \\ 
    (P,Q) & \mapsto \phi_\eps^{-1} ( P + \phi_\eps(Q)).
  \end{aligned}
\end{equation}
This action is independent of the choice of $\phi_\eps$ and hence
defined over $k$. The natural map $H^1(k,\J[2]) \to H^1(k,\J)$ is then
realised by $(\J_\eps,\pi_\eps) \mapsto \J_\eps$.
  
The isomorphism $\phi_\eps$ identifies the map $[-1]$ on $\J$ with an
involution $\iota_\eps$ on $\J_\eps$. Again it may be checked that
$\iota_\eps = \phi_\eps^{-1} \circ [-1] \circ \phi_\eps$ is
independent of the choice of $\phi_\eps$ and hence defined over $k$.

We have $\J/[-1] \isom \K \subset \PP^3$, where $\K$ is the Kummer
surface as described in Section~\ref{sec:kum}.  Twisting by
$\eps \in S^{(2)}(\J/k)$ gives
$\J_\eps/\iota_\eps \isom \K_\eps \subset \PP^3$ where $\K_\eps$ is a
twisted Kummer surface. (To recap Remark~\ref{rem:obs}: For general
$\eps \in H^1(k,\J[2])$ we might need to replace $\PP^3$ by a
Brauer-Severi variety $S$, but for Selmer group elements the
local-global principle for the Brauer group shows that
$S \isom \PP^3$.)

We record this set-up in a commutative diagram
\begin{equation}
\label{twist:J->K}
\begin{aligned}
 \xymatrix{ \J_\eps \ar[r] \ar[d]_{\phi_\eps} & \ar[d]^{\psi_\eps} \K_\eps \\
   \J \ar[r]^-{|2 \Theta|} & \K }
 \end{aligned}
\end{equation}
where $\phi_\eps$ is as in~\eqref{2cov} and
$\psi_\eps \in \GL_4(\kbar)$.  The divisor $H_\eps$ on $\J_\eps$
obtained by pulling back a hyperplane section on
$\K_\eps \subset \PP^3$ therefore satisfies
\begin{equation}
  \label{H-theta}
   H_\eps \sim \phi_\eps^*( 2 \Theta)
\end{equation}
where $\Theta$ is the theta divisor on $\J$.

In \cite[Chapter 3]{JialiThesis} the second author gave two different
methods for computing the twisted Kummer surface $\K_\eps$, both of
which worked by first computing the change of coordinates $\psi_\eps$.
This is typically defined over a degree $16$ number field.  Using the
models introduced in Section~\ref{sec:models} we are now able to
compute an equation for $\K_\eps$ without making any field extensions.
We start by finding a model $(G,H)$ that represents $\eps$.  We then
put $B = \Lambda \bH \Lambda^T$ where $\Lambda$ is defined
by~\eqref{def:Lambda}.  Then, exactly as in Section~\ref{sec:desing},
a quartic form defining $\K_\eps \subset \PP^3$ is obtained as the GCD
of the $3 \times 3$ minors of $B$. Moreover all the formulae we gave
in that section for $\K$, $\K^\vee$, $\dKum$, $\J$ and the maps
between them carry over immediately to their twisted counterparts
$\K_\eps$, $(\K_\eps)^\vee$, $\dKum_\eps$ and $\J_\eps$.  In
particular the double cover $\J_\eps \to \K_\eps$ is defined by
setting any leading $2 \times 2$ minor of $B$ equal to $-f_6$ times a
square.  As mentioned in the introduction, we find these formulae more
convenient to use than those in~\cite{2-coverings}.

\begin{Remark}
  \label{rem:MT}
  The model $(G,H)$ that represents $\eps$ also determines a
  $4 \times 4$ skew symmetric matrix $A = A_\eps$ as in
  Section~\ref{sec:recover}.  Just as in the untwisted case (see
  Section~\ref{sec:kum}) this matrix defines an isomorphism
  $\K_\eps \to (\K_\eps)^\vee$. In the context of testing equivalence
  of models, these matrices help us solve for the corresponding
  transformation when the models are equivalent.  This is analogous to
  the treatment of $3$-coverings of elliptic curves
  in~\cite{testeqtc}.  The matrices $A_\eps$ also give, by means of
  the formula~\eqref{MT}, the matrices $M_T$ describing the action of
  $T \in \J[2]$ on $\K_\eps \subset \PP^3$.  As explained in
  Remark~\ref{rem:redn} below, this is useful for reduction.
\end{Remark}

\begin{Remark}
  Let $c \in S^{(2)}(\J/k)$ be the canonical element defined in
  Section~\ref{sec:sel}. Then continuing with Remark~\ref{addc} we see
  that $(\K_\eps)^\vee \isom \K_{\eps + c}$.
\end{Remark}

The covering map $\pi_\eps : \J_\eps \to \J$ in \eqref{2cov} induces a
map $\K_\eps \to \K$, which we also call the covering map.  We now
explain how to compute this map explicitly.  Starting with a model
$(G,H)$ we use the formula in Proposition~\ref{lem:viadiff}(i) to
define quadratic forms $Q_0, \ldots, Q_5$. Then Lemma~\ref{lem:XiM}
gives a formula for the covering map from the twisted desingularised
Kummer $\dKum_\eps = \{ Q_3 = Q_4 = Q_5 = 0 \} \subset \PP^5$ to the
Kummer surface $\K$ written as a double cover of $\PP^2$ as in
\eqref{double_cover}.

We modify this to give the covering map $\K_\eps \to \K$ where both
surfaces are quartic surfaces in $\PP^3$. As suggested by
Remark~\ref{rem:Kmaps}(iii) our construction involves the $2 \times 2$
minors of $B = \Lambda \bH \Lambda^T$.  For
$Q = \sum c_{ijkl} z_{ij} z_{kl}$ a quadratic form we put
\[ Q \star H = \sum c_{ijkl} (B_{ik}B_{jl} - B_{il} B_{jk}). \]
Then we define quartic forms
$E_r = \adj( \Lambda \bG ( \bG^{-1} \bH )^r \Lambda^T)_{44}/x_4^2$
and
\begin{align*}
F_0 &= (2f_6)^{-1} E_1, \\
F_1 &= Q_2 \star H - 2 f_4 F_0, \\
F_2 &= -Q_1 \star H + f_3 F_0 = - E_2 - f_3 F_0, \\
F_3 &= Q_0 \star H, \\
F_4 &= f_6 E_3 - (f_2 Q_2 - f_3 Q_1 + f_4 Q_0) \star H 
              -( f_1 f_5 - 4 f_2 f_4 + 2 f_3^2) F_0.
\end{align*}

The quartic forms $F_i = F_i(H)$ are covariants of the model $H$.  By
this we mean that for some integers $d_0, \ldots, d_4$ we have
\[ F_i (\lambda H \circ \wedge^2 P) = ( \lambda \det P)^{d_i} F_i(H)
\circ P \]
for all $\lambda \in k^\times$ and $P \in \GL_4(k)$.
In fact $(d_0, \ldots, d_4) = (3,5,6,7,9)$. The next lemma
shows that these covariants define the covering map.
This generalises an observation of Weil~\cite{Weil-Hermite} in
the case of $2$-coverings of elliptic curves.

\begin{Lemma} 
\label{lem:quartics}
Let $\eps \in S^{(2)}(\J/k)$ be represented by a model $(G,H)$ and
let the $F_i = F_i(H)$ be the quartic polynomials defined above.
\begin{enumerate}
\item $\K_\eps \subset \PP^3$ has equation $F_0 = 0$.
\item Let the matrices $M_T$ be as defined in Remark~\ref{rem:MT}.
Then $F_0,\ldots,F_4$ are a basis for the space of
Heisenberg invariant quartics
\[ \qquad \{ F \in k[x_1, \ldots,x_4]_4 : F \circ M_T = \det(M_T) F 
  \text{ for all } T \in \J[2] \}. \]
\item The covering map $\K_\eps \to \K$ is given by
\[ (x_1: \ldots: x_4) \mapsto (F_1: \ldots: F_4). \]
\end{enumerate}
\end{Lemma}
\begin{proof}
  Since we are free to extend our field, and the $F_i$ are covariants,
  it suffices to prove this for $H$ as given in
  Section~\ref{sec:desing}.  We already checked (i) in
  Section~\ref{sec:desing}. For (ii) we checked by computer algebra
  that $F_0, \ldots, F_4$ are linearly independent and then used
  Lemmas~\ref{lem:heis} and~\ref{lem-heis} below. See
  Proposition~\ref{prop:quartic-factors} below for further details of
  the proof of (iii).
\end{proof}

\begin{Remark}
  For the purposes of Lemma~\ref{lem:quartics} it would not make any
  difference if we changed $F_1, \ldots, F_4$ by adding multiples of
  $F_0$.  The reason for making the choices we did is to simplify the
  statement of Proposition~\ref{prop:quartic-factors} below.
\end{Remark}


\section{The Cassels-Tate pairing on the $2$-Selmer group}
\label{sec:CTP}

Let $\A$ be an abelian variety defined over a number field $k$.  The
Cassels-Tate pairing is a bilinear map
\begin{equation}
\label{CTP1}
   \Sha(\A/k) \times \Sha(\A^\vee/k) \to \Q/\Z.
\end{equation}
It has the property that for any integer $n \geqslant 2$ the image of
multiplication-by-$n$ on $\Sha(\A/k)$ and kernel of
multiplication-by-$n$ on $\Sha(\A^\vee/k)$ are exact annihilators.
This follows from \cite[Chapter I, Theorem 6.13]{Milne} or more
directly from the injectivity of the middle vertical map in the
diagram on page 88 of \cite{Milne}. Flach \cite{Flach} showed that if
we identify $\A = \A^\vee$ using a principal polarisation, then the
pairing is skew-symmetric.

We take $\A = \J$ a genus $2$ Jacobian, and identify $\J = \J^\vee$
using the principal polarisation
\begin{equation}
\label{pp}
X \mapsto [\tau_X^* \Theta - \Theta].
\end{equation}
By composing with the map $S^{(2)}(\J/k) \to \Sha(\J/k)[2]$
in~\eqref{exseq-2desc}, the pairing~\eqref{CTP1} lifts to a pairing of
$\F_2$-vector spaces
\begin{equation}
\label{ctp2}
 \langle~,~\rangle_\CT : S^{(2)}(\J/k) \times S^{(2)}(\J/k) \to \F_2.
\end{equation}
Specialising the results cited in the last paragraph, we see that this
pairing is symmetric and bilinear and, by chasing around the
diagram~\eqref{commdiag1}, has kernel the image of the natural map
$S^{(4)}(\J/k) \to S^{(2)}(\J/k)$. The pairing is not in general
alternating, but rather has the following properties.
\begin{Lemma}[Poonen and Stoll]
\label{lem:PS}
Let $c \in S^{(2)}(\J/k)$ be the canonical element as defined in
Section~\ref{sec:sel}. Then
\begin{enumerate}
\item $\langle x,x+c \rangle_\CT = 0$ for all $x \in S^{(2)}(\J/k)$. 
\item $\langle c,c \rangle_\CT = 0$ if and only if the number of
  deficient places is even, where a place $v$ is {\em deficient}
  if $\C$ has no $k_v$-rational divisor of degree $1$.
\item The rank of the pairing~$\langle ~,~ \rangle_\CT$  
  is even if and only if $\langle c,c \rangle_\CT = 0$.
\end{enumerate}
\end{Lemma}
\begin{proof} (i) and (ii). See~\cite[Theorem~5 and Corollary~12]{PS}.
  Note that Poonen and Stoll
  write $c$ for what in our notation is the image of $c$ in $\Sha(\J/k)$. \\
  (iii) This follows by (i) and linear algebra.
\end{proof}
One consequence of Lemma~\ref{lem:PS} is that if the number of
deficient places is odd then the upper bound for $\rank \, \J(k)$
coming from $2$-descent improves by~$1$, without even having to
compute the pairing. Our focus in Section~\ref{sec:examples} is
therefore in giving examples where the rank bound improves by $2$ or
more.

\begin{Remark}
\label{rem:creutz}
Further partial information about the pairing is given by work of
Creutz \cite{Creutz} who (assuming there are no deficient places) gave
a criterion for whether the canonical element $c$ is in the kernel of
the pairing, equivalently, for whether the pairing on $S^{(2)}(\J/k)$
is alternating.
\end{Remark}

\begin{Remark} More generally, fixing a prime $p$, we may write $S_m$ for
  the image of the $p^m$-Selmer group in the $p$-Selmer group.  Then,
  exactly following Cassels' treatment~\cite{CasselsIII} in the case
  of elliptic curves, we have inclusions of $\F_p$-vector spaces
  \[\J(k)/2\J(k) \subset \ldots \subset S_3 \subset S_2 \subset S_1
    = S^{(p)}(\J/k)\]
  and a skew-symmetric pairing on each $S_m$ with kernel $S_{m+1}$.
  It is interesting to note (paraphrasing the work of
  Poonen and Stoll) that all of these pairings are alternating,
  except for the one we compute in this paper, that is, the pairing
  on $S_1$ when $p=2$. In particular, assuming finiteness
  of $\Sha(\J/k)[2^\infty]$, the rank bounds we compute (by a combination
  of $2$-descent and computing the pairing~\eqref{ctp2})
  always have the same parity as the true rank.
\end{Remark}

In Section~\ref{sec:def-ctp} we briefly review the definition of the
Cassels-Tate pairing~\eqref{CTP1} that is relevant to our work.  Then
in Section~\ref{sec:formula} we give our new formula for the
pairing~\eqref{ctp2}.  In outline it works as follows.  Given
$\eps,\eta \in S^{(2)}(\J/k)$ we compute equations for the twisted
Kummer surfaces $\K_\eps \subset \PP^3$ and $\K_\eta \subset \PP^3$.
If we can find a $k$-point on $\K_\eta$ then the image of $\eta$ in
$\Sha(\J/k)$ splits over a quadratic extension.  We can then represent
$\eta$ as a Brauer class on $\K_\eps$ split by the same quadratic
extension. We evaluate this Brauer class at suitable local points on
$\K_\eps$, and then compute $\langle \eps,\eta \rangle_{\CT}$ as the
sum of the local invariants.

To turn this formula into a practical method for computing the
pairing, we must compute a rational function on $\K_\eps$ that
represents the required Brauer class.  In the case of elliptic curves,
the first author's solution to this problem~\cite{bq-ctp}, was based
on the well known formulae relating the $x$-coordinates of the points
$P, Q, P+Q, P-Q$ on an elliptic curve.  The analogue of these formulae
in genus~$2$ are the biquadratic forms computed by
Flynn~\cite{biquadratic}. These describe remnants of the group law on
$\J$ on the Kummer surface $\K = \J/[-1]$.  This leads in
Section~\ref{sec:applic} to the definition of a $(2,2,2)$-form on
$\PP^3 \times \PP^3 \times (\PP^3)^\vee$, a twisted version of which
specialises to the rational function we need.  In
Sections~\ref{sec:heis}--\ref{sec:get222} we give equations for the
twisted $(2,2,2)$-form in terms of our invariant-theoretic formulae
for the $2$-covering map.  Interestingly, we discover that the
coefficients of the original $(2,2,2)$-form may be viewed as the
structure constants for an \'etale algebra.

\subsection{Definition of the Cassels-Tate pairing}
\label{sec:def-ctp}
We briefly review the homogeneous space definition
of the Cassels-Tate pairing~\eqref{CTP1}, following Poonen
and Stoll \cite{PS} and Milne \cite[Chapter 1, Remark 6.11]{Milne}.

We start with $a \in \Sha(\A/k)$ and $b \in \Sha(\A^\vee/k)$.  Let $X$
be a (locally trivial) principal homogeneous space over $k$
representing $a$. When we write $\Div^0(X)$ and $\Pic^0(X)$ we shall
mean these groups for $X$ over $\kbar$. Since $\Pic^0(X)$ is
canonically isomorphic as a $\Gal(\kbar/k)$-module to
$\Pic^0(\A) = \A^\vee(\kbar)$, the element
$b \in \Sha(\A^\vee/k) \subset H^1(k,\A^\vee)$ corresponds to an
element of $H^1(k,\Pic^0(X))$.

The short exact sequence of Galois modules
\[ 0 \to \kbar(X)^\times/\kbar^\times \to \Div^0(X) \to \Pic^0(X) \to
  0, \] induces a long exact sequence of Galois cohomology
\begin{equation}
\label{les1}
 \ldots \to H^1(k,\Pic^0(X)) \to H^2(k,\kbar(X)^\times/\kbar^\times)
 \to \ldots
\end{equation}
Let $b' \in H^2(k,\kbar(X)^\times/\kbar^\times)$ be the image of $b$.
Next we consider the short exact sequence
\[0 \to \kbar^\times \to \kbar(X)^\times \to \kbar(X)/\kbar^\times \to 0,\]
and its associated long exact sequence
\begin{equation}
\label{les2}
 \ldots \to H^2(k,\kbar^\times) \to
    H^2(k,\kbar(X)^\times) \to H^2(k,\kbar(X)^\times/\kbar^\times)
    \to H^3(k,\kbar^\times) \to \ldots
\end{equation}  
Since $k$ is a number field we have $H^3(k,\kbar^\times) = 0$.
Therefore $b'$ lifts to some $f' \in H^2(k,\kbar(X)^\times)$.  In the
local analogue of~\eqref{les2}, each $f'_v$ maps to zero (since $b$ is
locally trivial) and so is the image of some
$c_v \in H^2(k_v,\kbar_v^\times) = \Br(k_v)$.  The Cassels-Tate
pairing is defined by
\[ (a,b) \mapsto \sum_v \inv_v(c_v) \in \Q/\Z. \]

\begin{Remark}
  \label{rem-def-ctp}
  (i) Although the sum is over all places $v$ of $k$, it may be shown
  that only finitely many places contribute and so the
  sum is in fact finite. \\
  (ii) We can compute $c_v$ by evaluating $f'_v$ at any local point
  $P_v \in X(k_v)$. Since $b$ is locally trivial, the choice of $P_v$
  does not matter. \\
  (iii) The pairing is independent of the choice of $f'$ since by
  class field theory the sum of the local invariants of an element in
  $\Br(k)$ is always zero.
\end{Remark}

\subsection{A formula for the Cassels-Tate pairing}
\label{sec:formula}
In this section we give our new formula for the pairing~\eqref{ctp2}.
Accordingly, we start with $\eps, \eta \in S^{(2)}(\J/k)$ and aim to
compute $\langle \eps,\eta \rangle_\CT$.  We write $\K_\eps$,
$\K_\eta$ and $\K_{\eps+\eta}$ for the corresponding twisted Kummer
surfaces.  Our method depends on finding a rational point
$P \in \K_\eta(k)$.  In general there is no guarantee that such a
point exists, but as discussed in the introduction this does not
appear to be a severe restriction in practice.  Since
$\J_\eta \to \K_\eta$ is a double cover, the point $P$ lifts to a
point defined over $k(\sqrt{a})$ for some $a \in k$. We may suppose
that $P$ is not a node, and that $a$ is not a square, otherwise
$\J_\eta$ is trivial as a homogeneous space, in which case $\eta$ has
trivial image in $\Sha(\J/k)$ and $\langle \eps,\eta \rangle_\CT = 0$.
As in Section~\ref{sec:tw-kum}, we have
$\K_\eps = \J_\eps/\iota_\eps$.

\begin{Theorem}
\label{thm:ctp-formula}
Suppose that $P \in \K_\eta(k)$ lifts to a point on $\J_\eta$ defined
over $k(\sqrt{a})$.  Let $H_\eps$ and $H_{\eps+\eta}$ be divisors on
$\J_\eps$ and $\J_{\eps+\eta}$ obtained by pulling back hyperplane
sections on $\J_\eps \to \K_\eps \subset \PP^3$ and
$\J_{\eps+\eta} \to \K_{\eps+\eta} \subset \PP^3$.  Then
\begin{enumerate}
\item The point $P$ determines an isomorphism
  $\phi : \J_\eps \to \J_{\eps + \eta}$ defined over~$k(\sqrt{a})$.
\item There is a $k$-rational function $g$ on $\J_\eps$ with divisor
  \begin{equation*}
 \divisor(g) = \phi^* H_{\eps + \eta} + \iota_\eps^*(\phi^* H_{\eps + \eta})
  - 2 H_\eps. \end{equation*}
\item 
The Cassels-Tate pairing \eqref{ctp2} is given by
\[ \langle \eps,\eta \rangle_\CT = \sum_v (a,g(P_v))_v \] where for
each place $v$ of $k$ we pick a local point $P_v \in \J_\eps(k_v)$,
avoiding the zeros and poles of $g$, and
$(~,~)_v: k_v^\times/(k_v^\times)^2 \times k_v^\times/(k_v^\times)^2
\to \F_2$ is the Hilbert norm residue symbol.
\end{enumerate}
\end{Theorem}

\begin{Remark}
  \label{rem-g}
  (i) The theorem is only useful if we have a practical
  method for computing the rational function $g$.
  This is the subject of Sections~\ref{sec:applic}--\ref{sec:get222}. \\
  (ii) The scaling of $g$ is not important for
  the same reason as in Remark~\ref{rem-def-ctp}(iii). \\
  (iii) Since $g$ factors via $\K_\eps$ we may
  evaluate it at local points on $\K_\eps$, provided that we are
  careful to choose local points that can be lifted to local
  points on $\J_\eps$.
\end{Remark}

We start the proof of Theorem~\ref{thm:ctp-formula} by describing a
twisted form of the group law on $\J$.
\begin{Lemma}
\label{lem:twistedgplaw}
  Let $\eps,\eta \in H^1(k,\J[2])$. Let $\phi_\eps : \J_\eps \to \J$,
  $\phi_\eta : \J_\eta \to \J$ and
  $\phi_{\eps +\eta} : \J_{\eps + \eta} \to \J$ be isomorphisms
  defined over $\kbar$ making the diagram~\eqref{2cov} and its
  analogues for $\eta$ and $\eps + \eta$ commute.  Then, after
  possibly adjusting our choice of $\phi_{\eps +\eta}$, there is a
  morphism $\mu$ defined over $k$ making the following diagram commute
  \[ \xymatrix{ \J_\eps \,\,\, \times \,\,\, \J_\eta \ar[r]^-{\mu}
      \ar@<-1.5em>[d]^{\phi_\eps} \ar@<1.5em>[d]^{\phi_\eta}
      & \J_{\eps+\eta} \ar[d]^{\phi_{\eps+\eta}}  \\
      \J \,\,\, \times \,\,\, \J \ar[r]^-{+}
  &  \J \rlap{.} } \]
\end{Lemma}
\begin{proof}
  We know that $\eps$ is represented by a cocycle
  $(\sigma \mapsto \eps_\sigma)$ where
  $\sigma(\phi_\eps) \phi_\eps^{-1}$ is translation by
  $\eps_\sigma \in \J[2]$.  The corresponding statements hold for
  $\eta$ and $\eps + \eta$. Since the cocycles
  $(\sigma \mapsto (\eps + \eta)_\sigma)$ and
  $(\sigma \mapsto \eps_\sigma + \eta_\sigma)$ differ by a coboundary,
  we may adjust our choice of $\phi_{\eps +\eta}$ so that
  $(\eps + \eta)_\sigma = \eps_\sigma + \eta_\sigma$ for all
  $\sigma \in \Gal(\kbar/k)$.

  Let $\mu : \J_\eps \times \J_\eta \to \J_{\eps+\eta}$ be the
  morphism that makes the diagram in the statement of the lemma
  commute.  For any $P \in \J_\eps$, $Q \in \J_\eta$ and
  $\sigma \in \Gal(\kbar/k)$ we have
  \begin{align*}
    \sigma(\mu(P,Q)) & = \sigma(\phi_{\eps+\eta}^{-1}
                       ( \phi_\eps(P) + \phi_\eta(Q))) \\
     & = \phi_{\eps+\eta}^{-1} (\sigma( \phi_\eps(P)) + \sigma (\phi_\eta(Q)) -
                   \eps_\sigma - \eta_\sigma ) \\
     &= \phi_{\eps+\eta}^{-1} (\phi_\eps(\sigma P) + \phi_\eta(\sigma Q)) \\
     &= \mu( \sigma P, \sigma Q).
  \end{align*}                     
  This proves that $\mu$ is defined over $k$.  
\end{proof}  

\begin{Remark}
  \label{rem:choices}
  We are still free to replace $\phi_{\eps+\eta}$ by
  $P \mapsto \phi_{\eps+\eta}(P) + T$ for $T \in \J(k)[2]$, and for
  this reason there are $\#\J(k)[2]$ choices for the map $\mu$.
\end{Remark}

\begin{ProofOf}{Theorem~\ref{thm:ctp-formula}}
  Let $\{ Q,Q'\} \subset \J_\eta$ be the inverse image of
  $P \in \K_\eta$, and let $\mu$ be as defined in
  Lemma~\ref{lem:twistedgplaw}.  Then $\phi = \mu(-,Q)$ is an
  isomorphism $\J_\eps \to \J_{\eps+\eta}$ defined over
  $k(Q) = k(\sqrt{a})$.  This proves (i).  From
  Lemma~\ref{lem:twistedgplaw} we also get a commutative diagram
  \begin{equation}
    \label{comm-square}
    \begin{aligned}
       \xymatrix{ \J_\eps \ar[r]^-{\phi}
      \ar[d]_{\phi_\eps} 
      & \J_{\eps+\eta} \ar[d]^{\phi_{\eps+\eta}}  \\
      \J \ar[r]^-{\tau_R} &  \J \rlap{,} }
    \end{aligned}
  \end{equation}
  where $\tau_R$ is translation by $R = \phi_\eta(Q)$.
  
  Let $\psi : \J_\eta \to \J$ be the isomorphism given by
  $X \mapsto X - Q$. Then $\sigma(\psi) \psi^{-1}$ is translation by
  $Q - \sigma(Q)$. By~\eqref{phs} we have
  $Q - \sigma(Q) = \phi_\eta(Q) - \phi_\eta(\sigma(Q))$.  Since
  $\phi_\eta(Q) = R$ and $\iota_\eta$ swaps over $Q$ and $Q'$ we have
  $\phi_\eta(Q') = -R$. It follows that the image of $\eta$ in
  $H^1(k,\J)$ is represented by the cocycle
  \[ \sigma \mapsto \left\{ \begin{array}{cl}
   0_\J & \text{ if } \sigma(\sqrt{a}) = \sqrt{a}, \\
   2R & \text{ if } \sigma(\sqrt{a}) =  -\sqrt{a}.
   \end{array} \right. \]
                        
We follow the definition of the Cassels-Tate pairing in
Section~\ref{sec:def-ctp}, with $a$ and $b$ the images of $\eps$ and
$\eta$ in $\Sha(\J/k)$ and $\Sha(\J^\vee/k)$, in the latter case using
the principal polarisation~\eqref{pp}.  Now $b$ corresponds to the
element in $H^1(k,\Pic^0\J_\eps)$ represented by the cocycle
  \[ \sigma \mapsto \left\{ \begin{array}{ll}
             {\rm id} & \text{ if } \sigma(\sqrt{a}) = \sqrt{a}, \\
   \![ \phi_\eps^* (\tau_{2R}^* \Theta - \Theta) ]
   & \text{ if } \sigma(\sqrt{a}) = -\sqrt{a}. \end{array} \right. \]
Next we claim that
\[ \phi_\eps^* (\tau_{2R}^* \Theta - \Theta) \sim \phi_\eps^*
  (\tau_{R}^* (2\Theta) - 2 \Theta) \sim \phi^* H_{\eps + \eta} -
  H_\eps. \] Indeed the first linear equivalence follows from the fact
the polarisation~\eqref{pp} is a group homomorphism, whereas the
second follows from~\eqref{H-theta} and~\eqref{comm-square}.

To apply the connecting homomorphism in \eqref{les1} we lift to a
cochain
  \[ \sigma \mapsto \left\{ \begin{array}{ll}
             {\rm id} & \text{ if } \sigma(\sqrt{a}) = \sqrt{a}, \\
    \phi^* H_{\eps + \eta} - H_\eps & \text{ if } \sigma(\sqrt{a}) = -\sqrt{a},
    \end{array} \right. \]
and then take its differential to get
  \[ (\sigma_1, \sigma_2) \mapsto 
  \left\{ \begin{array}{ll}
  \phi^* H_{\eps + \eta} + \sigma_1(\phi^* H_{\eps + \eta}) - 2H_\eps
  & \text{ if } \sigma_1(\sqrt{a}) = \sigma_2(\sqrt{a}) = -\sqrt{a}, \\
   {\rm id} & \text{ otherwise. }  \end{array} \right. \]
For any $\sigma_1 \in \Gal(\kbar/k)$ with $\sigma_1(\sqrt{a}) = -\sqrt{a}$
we have $\sigma_1(\phi) = \iota_{\eps+\eta} \circ \phi \circ \iota_\eps$
and hence $\sigma_1(\phi^*H_{\eps + \eta}) = \iota_\eps^* (\phi^*H_{\eps + \eta})$.
Thus there exists a $k$-rational function $g$ on $\J_\eps$ as
specified in (ii), and we can take 
$c_v \in H^2(k_v,\kbar_v^\times) = \Br(k_v)$ to be the class of
  \[ (\sigma_1, \sigma_2) \mapsto \left\{ \begin{array}{cl}
  g(P_v) & \text{ if } \sigma_1(\sqrt{a}) = \sigma_2(\sqrt{a}) = -\sqrt{a}, \\
   1 & \text{ otherwise, } \end{array} \right. \]
where $P_v \in \J_\eps(k_v)$ is any local point avoiding the zeros
and poles of $g$.

By \cite[Chapter XIV, Section 2, Proposition 5]{Serre-LF} the element
$c_v \in \Br(k_v)$ is represented by the quaternion algebra
$(a,g(P_v))$. In particular its local invariant is $0$ or $1/2$
according as the Hilbert symbol $(a,g(P_v))_v$ is $1$ or $-1$.  By
abuse of notation, we let the Hilbert symbol take values in $\F_2$.
The Cassels-Tate pairing is then given by summing the Hilbert symbols
$(a,g(P_v))_v$ as claimed.
\end{ProofOf}

\subsection{Remnants of the group law}
\label{sec:applic}

In this section we explain how the rational function $g$ in
Theorem~\ref{thm:ctp-formula} can be computed by specialising a
certain $(2,2,2)$-form. In later sections we explain how to compute
this $(2,2,2)$-form.

The Kummer surface $\K = \J/[-1]$ retains some remnants of the group
law on $\J$. As usual we embed $\K \subset \PP^3$ with coordinates
$x_1, \ldots, x_4$.  Flynn \cite{biquadratic} computed biquadratic
forms $\Phi_{ij}$ such that for all $P,Q \in \J$ we have
\begin{equation}
\label{remnant}
\begin{aligned}
  \Phi_{ij} (x_1(P), &\ldots, x_4(P);x_1(Q), \ldots, x_4(Q)) \\
 & \propto \big(  x_i(P+Q) x_j(P-Q) + x_i(P-Q) x_j(P+Q) \big)
\end{aligned}
\end{equation}
where the constant of proportionality is independent of
$1 \leqslant i,j \leqslant 4$.  We package these biquadratic forms as
a $(2,2,2)$-form on $\PP^3 \times \PP^3 \times
(\PP^3)^\vee$. Explicitly, we put $\x=(x_1,\ldots,x_4)$,
$\y=(y_1, \ldots,y_4)$, $\z = (z_1, \ldots,z_4)$ and define
\[ \Phi(\x,\y,\z) = \sum_{i,j=1}^4 \Phi_{ij}(\x,\y) z_i z_j. \] Since
the $\Phi_{ij}$ are symmetric, i.e.,
$\Phi_{ij}(\x,\y) = \Phi_{ij}(\y,\x)$, we have
$\Phi(\x,\y,\z) = \Phi(\y,\x,\z)$.  If $P,Q \in \J$ then
by~\eqref{remnant} we have
\begin{equation}
\label{eqn:prod1}
\begin{aligned}
  \Phi(x_1(P), \ldots, &x_4(P); x_1(Q), \ldots, x_4(Q);c_1,c_2,c_3,c_4) \\
 & \propto \big( \textstyle\sum_{i=1}^4 c_i x_i(P+Q) \big)
 \big( \textstyle\sum_{j=1}^4 c_j x_j(P-Q) \big).
\end{aligned}
\end{equation}
where the constant of proportionality is independent of
$c_1, \ldots ,c_4$.  We note for later use the immediate consequence
that for fixed $Q \in \J$ and $(c_1: \ldots: c_4) \in (\PP^3)^\vee$
the left hand side of~\eqref{eqn:prod1} is a non-zero function of $P$.

Just as the $(2,2,2)$-form $\Phi$ corresponds to the group law on
$\J$, there is a twisted $(2,2,2)$-form $\Phi_{\eps,\eta}$
corresponding to the twisted group law
\[ \mu : \J_\eps \times \J_\eta \to \J_{\eps + \eta} \]
defined in Lemma~\ref{lem:twistedgplaw}.
The analogue of~\eqref{eqn:prod1} is that if
$P \in \J_\eps$ and $Q \in \J_\eta$ then
\begin{equation}
\label{eqn:prod2}
\begin{aligned}
  \Phi_{\eps,\eta}(x_1(P),\ldots, x_4&(P); x_1(Q), \ldots,
   x_4(Q);c_1,c_2,c_3,c_4) \\
 & \propto \big( \textstyle\sum_{i=1}^4 c_i x_i(\mu(P,Q)) \big)
  \big( \textstyle\sum_{j=1}^4 c_j x_j(\mu(P,\iota_\eta (Q))) \big).
\end{aligned}
\end{equation}

The next theorem computes the rational function $g$ in
Theorem~\ref{thm:ctp-formula}.
\begin{Theorem}
  \label{thm:getg}
  Let $Q \in \J_\eta$ and let
  $\phi = \mu(-,Q) : \J_\eps \to \J_{\eps+\eta}$.  Let $H_\eps$ and
  $H_{\eps+\eta}$ be the divisors on $\J_\eps$ and $\J_{\eps+\eta}$
  obtained by pulling back the hyperplane sections $\{x_1 = 0\}$ and
  $\{c_1 x_1 + \ldots + c_4 x_4 = 0\}$ on
  $\J_\eps \to \K_\eps \subset \PP^3$ and
  $\J_{\eps+\eta} \to \K_{\eps+\eta} \subset \PP^3$.  Then
  \[ g  =
    \Phi_{\eps,\eta}(x_1, \ldots, x_4; x_1(Q), \ldots, x_4(Q);
    c_1,c_2,c_3,c_4)/x_1^2 \]
  is a rational function on $\J_\eps$ with divisor
    \begin{equation}
 \divisor(g) = \phi^* H_{\eps + \eta} + \iota_\eps^* (\phi^* H_{\eps + \eta})
  - 2 H_\eps. \end{equation}
\end{Theorem}

\begin{proof} Let $P \in \J_\eps$ and
  $P' = \phi(P) = \mu(P,Q) \in \J_{\eps + \eta}$.
  If $c_1 x_1(P') + \ldots
  + c_4 x_4 (P') = 0$ then by~\eqref{eqn:prod2} we have $g(P)=0$.
  Therefore $g$ vanishes on $\phi^* H_{\eps + \eta}$. Since $g$
  factors via the double cover $\J_\eps \to \K_\eps$ it also
  vanishes on $\iota_\eps^* (\phi^* H_{\eps + \eta})$.

  By the observation following~\eqref{eqn:prod1} we know
  that $g$ does not vanish identically on $\J_\eps$.
  Assuming that $\phi^* H_{\eps + \eta}$
  and $\iota_\eps^* (\phi^* H_{\eps + \eta})$
  are distinct and irreducible it follows that
  \begin{equation}
    \label{withE}
    \divisor(g) = \phi^* H_{\eps + \eta} + \iota_\eps^* (\phi^* H_{\eps + \eta})
    + E - 2 H_\eps \end{equation}
  for some effective divisor $E$. By considering the quotient of $g$ and
  the rational function of the same name
  in Theorem~\ref{thm:ctp-formula}(ii), it follows that $E = 0$.

  This completes the proof under the assumption in the last paragraph.
  We showed in \cite[Section 5.2.3]{JialiThesis} that the assumption
  is satisfied for generic choices of $Q \in \J$ and
  $(c_1: \ldots :c_4) \in (\PP^3)^\vee$, and from this we were able to
  deduce that the theorem holds in general.
\end{proof}

In \cite{JialiThesis} the second author computed the twisted
$(2,2,2)$-form $\Phi_{\eps,\eta}$ from the untwisted $(2,2,2)$-form
$\Phi$ by directly making the changes of coordinates $\psi_\eps$,
$\psi_\eta$ and $\psi_{\eps+\eta}$ as defined in
\eqref{twist:J->K}. Since these $4 \times 4$ matrices are typically
defined over (different) degree $16$ number fields, this method is
rather slow. Our improved method for computing the twisted Kummer
surfaces (see Section~\ref{sec:tw-kum}) also means we do not have
these matrices to hand, although it would be possible to compute them
in hindsight by matching up the nodes on $\K$ and $\K_\eps$.

In Section~\ref{sec:get222} we give a formula, analogous to that
in~\cite{bq-ctp}, that enables us to compute the twisted
$(2,2,2)$-form $\Phi_{\eps,\eta}$ without working in field extensions
of such large degree.

\subsection{The Heisenberg group}
\label{sec:heis}
In this section we describe some of the geometry behind
our method for computing the $(2,2,2)$-forms.

\begin{Definition}
\label{def:heis}
  The {\em standard level $2$ Heisenberg group} is
the subgroup $H_2 \subset \GL_4$ generated by the matrices
    \[
      \begin{pmatrix}
        1 & 0 & 0 & 0 \\
        0 & -1 & 0 & 0 \\
        0 & 0 & 1 & 0 \\ 
        0 & 0 & 0 & -1
      \end{pmatrix}, \,\,\,\,
      \begin{pmatrix}
        1 & 0 & 0 & 0 \\
        0 & 1 & 0 & 0 \\
        0 & 0 & -1 & 0 \\ 
        0 & 0 & 0 & -1
      \end{pmatrix}, \,\,\,\,
      \begin{pmatrix}
        0 & 1 & 0 & 0 \\
        1 & 0 & 0 & 0 \\
        0 & 0 & 0 & 1 \\ 
        0 & 0 & 1 & 0
      \end{pmatrix}, \,\,\,\,
     \begin{pmatrix}
        0 & 0 & 1 & 0 \\
        0 & 0 & 0 & 1 \\
        1 & 0 & 0 & 0 \\ 
        0 & 1 & 0 & 0
      \end{pmatrix}.
    \]
    It is a non-abelian group of order $2^5$ and sits in an exact sequence
    \[ 0 \to \{\pm 1\} \to H_2 \to \overline{H}_2 \to 0, \]
    where $\overline{H}_2 \isom (\Z/2\Z)^4$ is the image of $H_2$ in $\PGL_4$.
\end{Definition}
    
\begin{Lemma}
  \label{lem:heis}
  Let $\K \subset \PP^3$ be a Kummer surface defined over an algebraically
  closed field of characteristic not $2$. We may change coordinates
  on $\PP^3$ so that
  \begin{enumerate}
  \item the isomorphisms $\K \to \K$ corresponding to a
    $(2,4)$-partition of the Weierstrass points (equivalently to
    translation by some $0 \not= T \in \J[2]$) are given by the
    non-trivial elements of $\overline{H}_2$, and
  \item the isomorphisms $\K \to \K^\vee$ corresponding to a $(1,5)$,
    respectively $(3,3)$, partition of the Weierstrass points are
    given by the skew-symmetric, respectively symmetric, matrices in
    $\overline{H}_2$.
  \end{enumerate}
\end{Lemma}
\begin{proof}
  The existence of a change of coordinates satisfying (i) is a special
  case of \cite[Example 6.7.4]{BirkenhakeLange}.
  Since all the matrices are orthogonal, the conclusion in (i) for
  $\K$ also holds for the dual Kummer $\K^\vee$. The claims in (ii)
  follow by Remark~\ref{rem:6-10} and the fact that $\overline{H}_2$
  is its own centraliser inside $\PGL_4$.
\end{proof}
  
\begin{Lemma}
\label{lem-tr-id}
Let $\K \subset \PP^3$ be a Kummer surface defined over a field of
characteristic not $2$. Let $B_1, \ldots, B_{10}$ be $4 \times 4$
symmetric matrices whose images in $\PGL_4$ correspond to the
$(3,3)$-partitions of the Weierstrass points. If we define quadratic
forms by $Q_i(x_1, \ldots, x_4) = \frac{1}{2} \x^T B_i \x$ and
$Q_i^*(y_1, \ldots, y_4) = \frac{1}{2} \y^T B_i^{-1} \y$ then
\begin{equation}
\label{trace-identity}
      \sum_{i=1}^{10} Q_i(x_1, \ldots, x_4) Q_i^*(y_1, \ldots, y_4) =
      \left(\sum_{j=1}^4 x_j y_j\right)^2.
 \end{equation}
\end{Lemma}
\begin{proof}
  The conclusions of the lemma are unchanged if we extend our field,
  rescale the matrices $B_i$, or make a change of coordinates on
  $\PP^3$.  By Lemma~\ref{lem:heis} we may therefore assume that the
  $B_i$ are (lifts to $\GL_4$ of) the symmetric matrices in
  $\overline{H}_2$.  The left hand side of~\eqref{trace-identity} is
  now
  \begin{align*} \frac{1}{4} \sum_{u,v \in \{\pm 1\}}
      (x_1^2 + u x_2^2 & + v x_3^2 + uv x_4^2)
      (y_1^2 + u y_2^2 + v y_3^2 + uv y_4^2) \\
    &+ \sum_{u \in \{ \pm 1 \}} (x_1 x_2 + u x_3 x_4) (y_1 y_2 + u y_3 y_4) \\  
    &+ \sum_{u \in \{ \pm 1 \}} (x_1 x_3 + u x_2 x_4) (y_1 y_3 + u y_2 y_4)  \\
    &+ \sum_{u \in \{ \pm 1 \}} (x_1 x_4 + u x_2 x_3) (y_1 y_4 + u y_2 y_3).
  \end{align*}
  Expanding this out gives the right hand side of~\eqref{trace-identity}.
\end{proof}
  
\begin{Lemma}
\label{lem-heis}
The space of quartic forms invariant under the natural action of the
standard level $2$ Heisenberg group $H_2$ is $5$-dimensional, spanned
by
\begin{align*}
    x_1^4 + x_2^4 + x_3^4 + x_4^4, \quad x_1^2 x_2^2 + x_3^2 x_4^2, \quad
    x_1^2 x_3^2 + x_2^2 x_4^2,  \quad x_1^2 x_4^2 + x_2^2 x_3^2, \quad
    x_1 x_2 x_3 x_4.
\end{align*}
Up to scalars, there are exactly $10$ quartic forms in this space that
factor as the square of a quadratic form. These are the squares of the
quadratic forms corresponding to the symmetric matrices in
$\overline{H}_2$.
\end{Lemma}
\begin{proof}
  The lemma may be proved either by a brute force calculation, or by
  decomposing the 2nd and 4th symmetric powers of the standard
  representation of $H_2$ into irreducible representations.
\end{proof}

\subsection{A formula for the untwisted $(2,2,2)$-form}
\label{sec:untwisted}

As usual let $\C$ be the genus $2$ curve with equation $y^2 = f(x)$,
where the coefficients of $f$ are labelled as in~\eqref{eqn:C}.  We
recall that Flynn~\cite{biquadratic} computed the biquadratic forms
$\Phi_{ij}(\x,\y)$, with coefficients in $\Z[f_0, \ldots, f_6]$, and
in Section~\ref{sec:applic} we packaged these as a $(2,2,2)$-form
\begin{equation}
  \label{display:222}
  \Phi(\x,\y,\z) = \sum_{i,j=1}^4 \Phi_{ij}(\x,\y) z_i z_j.
\end{equation}
A generic calculation shows that
\begin{equation}
  \label{222->kum}
  \Phi(\x;\x;z_1,z_2,z_3,0) = 2 (z_2^2 - z_1 z_3) F(\x)
\end{equation}
where $F$ is the equation~\eqref{kumeqn} for $\K \subset \PP^3$.

In this section we give a new formula for the $(2,2,2)$-form $\Phi$,
and in the next section we modify it to give a formula for the twisted
$(2,2,2)$-form $\Phi_{\eps,\eta}$. Both formulae involve working with
the \'etale algebra of degree $10$, say $L_{10}$, corresponding to the
$(3,3)$-partitions of the Weierstrass points.

\begin{Remark}
  As explained in \cite[Lemma 5.3]{Stoll}, after possibly replacing
  $f(x)$ by $f(x+n)$ for some $n \in \Z$, we can take $L_{10}$ to be
  defined by the degree $10$ polynomial specified in \cite[page
  56]{CF}.  If $L_{10}$ is a field and $f$ has roots
  $\theta_1, \ldots, \theta_6$ then this is the minimal polynomial of
  $\theta_1 \theta_2 \theta_3 + \theta_4 \theta_5 \theta_6$.
\end{Remark}

Supposing that $f$ factors as the product of two cubics, say,
\begin{equation}
\label{fact33}
  f(x) = f_6 (x^3 + r_2 x^2 + r_1 x + r_0)
  (x^3 + s_2 x^2 + s_1 x + s_0),
\end{equation}
we define a column vector
\begin{equation}
  \label{def:la}
  \la = (\la_1, \ldots,\la_4)^T = (f_6 (r_0 s_2 + r_2 s_0), f_6 (r_0 + s_0),
    f_6 (r_1 + s_1), 1 )^T \end{equation}
and a symmetric matrix
  \[ A = f_6 \begin{pmatrix}
   r_2 - s_2  &                -r_1 + s_1  &           r_0 - s_0 \\
   -r_1 + s_1  & r_0 + r_1 s_2 - r_2 s_1 - s_0  &    -r_0 s_2 + r_2 s_0 \\
   r_0 - s_0  &          -r_0 s_2 + r_2 s_0  &     r_0 s_1 - r_1 s_0
    \end{pmatrix}. \]

\begin{Lemma}
  \label{lem:strconsts}
  If $f$ factors as in~\eqref{fact33} as the product of two cubic
  polynomials, then the $(2,2,2)$-form $\Phi$ has the following
  properties.
  \begin{enumerate}
  \item $\Phi(\x,\y,\lambda) 
    =  (\x^T B \y)^2$ where
    \[ B = \la \la^T +
      \begin{pmatrix} \adj A & 0 \\ 0 & 0 \end{pmatrix}. \]
  \item If we put $b_{ii} = B_{ii}$ and $b_{ij} = 2 B_{ij}$ for
    $i \not= j$ then
    \[ b_{ij} b_{kl} = \sum_{m \leqslant n} {\rm coeff}(\Phi_{mn}| x_i
      x_j y_k y_l) b_{mn}\] for all $1 \leqslant i,j,k,l \leqslant 4$.
  \end{enumerate}
\end{Lemma}
\begin{proof}
  We checked this by computer algebra, using Flynn's formulae for the
  biquadratic forms $\Phi_{ij}(\x,\y)$ specialised to $f$ of the
  form~\eqref{fact33}.
\end{proof}

\begin{Remark}
  The $b_{ij}$ are unchanged when we swap over the two cubic factors
  in~\eqref{fact33}, and are therefore elements of $L_{10}$.  In fact
  they form a basis.  Lemma~\ref{lem:strconsts}(iii) then says that
  the $10^3$ coefficients of the $(2,2,2)$-form $\Phi$ are the
  structure constants for the degree $10$ \'etale algebra $L_{10}$
  with respect to this basis.
\end{Remark}

\begin{Proposition}
  \label{prop:222}
  Let $\lambda_1, \lambda_2, \lambda_3 \in L_{10}$ be given
  by~\eqref{def:la}.  Then the quartic form
  \[ P(x_1,\ldots,x_4) =2 (\lambda_2^2 - \lambda_1 \lambda_3) F(\x) +
    2 \sum_{i=1}^3 \lambda_i \Phi_{i4}(\x,\x) + \Phi_{44}(\x,\x) \]
  factors as $P = Q^2$ where $Q \in L_{10}[x_1,x_2,x_3,x_4]$ is a
  quadratic form.  If we scale $Q$ so that the coefficient of $x_4^2$
  is $1$ and write $Q^*$ for the dual quadratic form (in the sense of
  Lemma~\ref{lem-tr-id}) then the $(2,2,2)$-form is given by
  \[ \Phi(\x,\y,\z) = \tr_{L_{10}/k} (Q(\x) Q(\y) Q^*(\z)). \]
\end{Proposition}
\begin{proof}
  By~\eqref{display:222} and~\eqref{222->kum} we have
  $\Phi(\x,\x,\lambda) = P(x_1, \ldots,x_4)$.  By
  Lemma~\ref{lem:strconsts}(i) this factors as $P = Q^2$ where
  $Q(\x) = \sum_{i \leqslant j} b_{ij} x_i x_j$ and the coefficient of
  $x_4^2$ is $1$.  By Lemma~\ref{lem:strconsts}(ii) we have
  \begin{equation}
  \label{mult-eqn}
  Q(\x) Q(\y) = \left( \sum_{i \leqslant j} b_{ij} x_i x_j \right)
  \left( \sum_{k \leqslant l} b_{kl} y_k y_l \right) =
  \sum_{m \leqslant n} \Phi_{mn}(\x,\y) b_{mn}.
  \end{equation}
  On the other hand by Lemma~\ref{lem-tr-id} we have
  \begin{equation}
  \label{dual-eqn}
   \tr_{L_{10}/k} ( b_{mn} Q^*(\z)) = \left\{ \begin{array}{ll}
  z_m z_n & \text{ if } m = n, \\
  2 z_m z_n & \text{ if } m \not= n. \end{array} \right.
  \end{equation}
  Combining~\eqref{mult-eqn} and~\eqref{dual-eqn} gives
  \[  \tr_{L_{10}/k} (Q(\x) Q(\y) Q^*(\z)) = \sum_{m,n=1}^4 \Phi_{mn}(\x,\y)
    z_m z_n =  \Phi(\x,\y,\z). \qedhere \]
\end{proof}

\begin{Remark}
  As noted in \cite[page~24]{CF}, the doubling map on the Kummer
  surface $\K = \{ F = 0 \} \subset \PP^3$ is given by
  \[ (x_1 : x_2 : x_3 : x_4) \mapsto (\Phi_{14}(\x,\x) :
    \Phi_{24}(\x,\x) : \Phi_{34}(\x,\x) : \tfrac{1}{2}
    \Phi_{44}(\x,\x)). \] The quartic form $P$ in
  Proposition~\ref{prop:222} is a linear combination of the Heisenberg
  invariant quartics $F, \Phi_{14}(\x,\x), \ldots, \Phi_{44}(\x,\x)$
  that factors as the square of a quadratic form.  By
  Lemma~\ref{lem-heis} this quadratic form $Q$ corresponds to a
  $(3,3)$-partition of the Weierstrass points. In particular, an
  alternative way to compute $Q$ is as the quadratic form defined by
  the symmetric matrix~\eqref{3-3}.
\end{Remark}

\subsection{A formula for the twisted $(2,2,2)$-form}
\label{sec:get222}
In the last section we gave a formula for the untwisted $(2,2,2)$-form
$\Phi$.  We now modify this to compute the twisted $(2,2,2)$-form
$\Phi_{\eps,\eta}$.

\begin{Proposition}
\label{prop:quartic-factors}
Let $\lambda_1, \ldots, \lambda_4 \in L_{10}$ be given
by~\eqref{def:la}. Let $\eps \in S^{(2)}(\J/k)$ be represented by a
model $(G,H)$, and let $F_0, \ldots, F_4 \in k[x_1, \ldots,x_4]$ be
the associated quartic forms, as defined in Section~\ref{sec:tw-kum}.
Then the quartic form
\begin{equation}
  \label{def:P}
  P_\eps(x_1, \ldots,x_4)  = (\lambda_2^2 - \lambda_1 \lambda_3)
  F_0(x_1, \ldots,x_4) + \sum_{i=1}^4 \lambda_i F_i(x_1, \ldots,x_4)
\end{equation}
factors as
\begin{equation}
\label{eqn:alpha}
P_\eps(x_1, \ldots, x_4) = \alpha_\eps \, Q_\eps(x_1, \ldots, x_4)^2
\end{equation}
where $\alpha_\eps \in L_{10}^\times$ and
$Q_\eps \in L_{10} [x_1, \ldots, x_4]$ is a quadratic form.
\end{Proposition}
\begin{proof}
  For the proof we are free to extend our field $k$.  Since the $F_i$
  are covariants we may then reduce to the case where $H$ is as given
  in Section~\ref{sec:desing}.  In this case we find that
  $F_0, \ldots, F_4$ are (up to an overall scaling) the quartics
  \[ 2F, \,\, 2\Phi_{14}(\x,\x), \,\, 2\Phi_{24}(\x,\x), \,\,
    2\Phi_{34}(\x,\x), \,\, \Phi_{44}(\x,\x). \] We are done by the
  first part of Proposition~\ref{prop:222}.
\end{proof}

\begin{Theorem}
  \label{thm:twisted222}
  Let $\eps,\eta \in S^{(2)}(\J/k)$.  The twisted $(2,2,2)$-form is
  given by
  \[ \Phi_{\eps,\eta}(\x,\y,\z) = \tr_{L_{10}/k} (\mu Q_\eps(\x)
    Q_\eta(\y) Q_{\eps+\eta}^*(\z)) \] for some $s \in k^\times$ and
  $\mu \in L_{10}^\times$ satisfying
  $\alpha_\eps \alpha_\eta/ \alpha_{\eps+\eta} = s \mu^2$.
\end{Theorem}

\begin{proof}
  Let $\psi_\eps \in \GL_4(\kbar)$ be the matrix in~\eqref{twist:J->K}
  scaled so that the 
  quartic forms in
  Propositions~\ref{prop:222} and~\ref{prop:quartic-factors} are
  related by $P \circ \psi_\eps = P_\eps$. The corresponding
  quadratic forms are then related by 
  $Q \circ \psi_\eps = \sqrt{\alpha_\eps} \, Q_\eps$. We have the same
  relations with $\eps$ replaced by $\eta$ or $\eps+\eta$.
  
  We write $\tr$ for the trace from
  $\overline{L}_{10} = L_{10} \otimes_k \kbar$ to
  $\kbar$.  First by definition and then using
  Proposition~\ref{prop:222}, the twisted $(2,2,2)$-form
  $\Phi_{\eps,\eta}$ is a scalar multiple of
  \begin{align*}
    \Phi(\psi_\eps \, \x , \psi_\eta \, \y, \psi_{\eps+\eta}^{-T} \, \z)
    &= \tr \big( (Q \circ \psi_\eps) (\x) \,(Q \circ \psi_\eta) (\y) \,
      (Q \circ \psi_{\eps + \eta})^* (\z) \big) \\
    &= \tr \big( \sqrt{\alpha_\eps \alpha_\eta /\alpha_{\eps+\eta}} \,
      Q_\eps (\x) \, Q_\eta (\y) \, Q_{\eps + \eta}^* (\z) \big).
  \end{align*}
  Viewing the trace as a sum over conjugates, this last formula writes
  $\Phi_{\eps,\eta}$ as a linear combination of 10 $(2,2,2)$-forms.
  It may be checked using Lemma~\ref{lem:heis} that these 10 forms
  are linearly independent.  Since $\Phi_{\eps,\eta}$ is defined over
  $k$ this forces
  $\sqrt{\alpha_\eps \alpha_\eta/ \alpha_{\eps+\eta}} = s_0 \mu$ for
  some $s_0 \in \kbar^\times$ and $\mu \in L_{10}^\times$.  Squaring
  both sides then shows that $s_0^2 \in k$.
\end{proof}

\begin{Remark}
\label{rem:S6}
Theorem~\ref{thm:twisted222} is not always sufficient to compute the
twisted $(2,2,2)$-form since the equation
$\alpha_\eps \alpha_\eta/ \alpha_{\eps+\eta} = s \mu^2$ does not
determine $s$ and $\mu$ uniquely.  This is not a problem if for
example $\Gal(f) = S_6$.  In this case $L_{10}$ is a field, and it has
no quadratic subfields.  It follows that $\mu$ is uniquely determined
up to multiplication by an element of $k^\times$. This is sufficient
for computing the twisted $(2,2,2)$-form, again up to multiplication
by an element of $k^\times$, but as explained in
Remark~\ref{rem-g}(ii) this scaling is not important.
\end{Remark}

In order to handle arbitrary Galois actions on the Weierstrass points
it remains to compute $\mu$ in a more systematic way. Our answer
depends on the scaling of the quadratic forms $Q_\eps$, $Q_\eta$ and
$Q_{\eps+\eta}$, and indeed such a dependence is necessary since
in~\eqref{eqn:alpha} we are free to multiply $Q_\eps$ by any element
of $\mu_2(L_{10})$ without changing $\alpha_\eps$.

Let $W$ be the set of Weierstrass points, and $\Gamma$ the set of
$(3,3)$-partitions of~$W$.  Let
$X = \{ x \in \Map(W,\mu_2) : \prod_{\theta \in W} x(\theta) = 1 \}$.
There is a commutative diagram with exact rows
\begin{equation*}
  \begin{aligned}
 \xymatrix{ 0 \ar[r] & \mu_2 \ar[r] \ar@{=}[d] &
    X \ar[r] \ar[d]_{\iota} & \J[2] \ar[d] \ar[r] & 0 \\
 0 \ar[r] & \mu_2 \ar[r] & \Map(\Gamma,\mu_2)
 \ar[r] & \Map(\Gamma,\mu_2)/\mu_2 \ar[r] & 0 }
\end{aligned}
\end{equation*}
where the middle vertical map is given by
$\iota(x)(\gamma_1 | \gamma_2) = \prod_{\theta \in \gamma_1}
x(\theta).$ Taking Galois cohomology gives another commutative diagram
with exact rows
\begin{equation}
  \label{cd}
  \begin{aligned}
    \xymatrix{ k^\times/(k^\times)^2 \ar[r] \ar@{=}[d] & H^1(k,X)
      \ar[r] \ar[d]_{\iota_*} & H^1(k,\J[2]) \ar[d] \ar[r]^-{\cup c}
      & \Br(k) \ar@{=}[d] \\
      k^\times/(k^\times)^2 \ar[r] & L_{10}^\times/(L_{10}^\times)^2
      \ar[r] & H^1(k,\mu_2(\overline{L}_{10})/\mu_2) \ar[r] & \Br(k) }
  \end{aligned}
\end{equation}
We identify
\[ H^1(k,X) \isom \frac{\{ (\xi,m)
    \in L^\times \times k^\times
    | N_{L/k}(\xi) = m^2 \}}{ \{(\nu^2, N_{L/k}(\nu)) |
        \nu \in L^\times \}}. \]
so that the top row (where the first map is now $r \mapsto (r,r^3)$)
gives us the isomorphism~\eqref{xi-m-pairs}.
      
\begin{Definition}
  \label{def:w}
  For $(\xi,m) \in  L^\times \times k^\times$ with $N_{L/k}(\xi) = m^2$,
  let $w(\xi,m) \in L_{10} = \Map_k(\Gamma,\kbar)$ be given by
  \[ (\theta_1\theta_2\theta_3 | \theta_4 \theta_5 \theta_6) \mapsto
    \xi(\theta_1) \xi(\theta_2) \xi(\theta_3) + \xi(\theta_4)
    \xi(\theta_5) \xi(\theta_6) + 2m \] and likewise under all
  permutations of $\theta_1, \ldots, \theta_6$.
\end{Definition}

\begin{Lemma}
  \label{lem:w}
  The map $\iota_* : H^1(k,X) \to L_{10}^\times/(L_{10}^\times)^2$
  sends $(\xi,m) \mapsto w(\xi,m)$, provided that $w(\xi,m)$ is a
  unit.
\end{Lemma}
\begin{proof}
  Let $(\xi,m) \in L^\times \times k^\times$ with
  $N_{L/k}(\xi) = m^2$, and let $w = w(\xi,m)$.  We pick square roots
  $\sqrt{\xi(\theta)}$ in such a way that
  $\prod_{\theta \in W} \sqrt{\xi(\theta)} = m$. We can then take
  $\sqrt{\xi} \in \Map(W,\kbar) = L \otimes_k \kbar$ to be given by
  $\theta \mapsto \sqrt{\xi(\theta)}$, and
  $\sqrt{w} \in \Map(\Gamma,\kbar) = L_{10} \otimes_k \kbar$ to be
  given by
  \[ (\gamma_1 | \gamma_2 ) \mapsto \prod_{\theta \in \gamma_1}
    \sqrt{\xi(\theta)} + \prod_{\theta \in \gamma_2}
    \sqrt{\xi(\theta)} \] We find that
  $\sigma \sqrt{w}/\sqrt{w} = \iota (\sigma \sqrt{\xi}/\sqrt{\xi})$
  for all $\sigma \in \Gal(\kbar/k)$, and from this the lemma follows.
\end{proof}

\begin{Proposition}
  \label{lem:rank1}
  Let $\eps \in S^{(2)}(\J/k)$ be represented by a model $(G,H)$.  Let
  $\bG$ and $\bH$ be the $6 \times 6$ symmetric matrices corresponding
  to $G$ and $H$, and let $S$ be the quadratic form defined by
  $\bH \bG^{-1} \bH$.  Let $Z$ and $W$ be the $4 \times 4$ skew
  symmetric matrices defined in~\eqref{def:ZW}, and let $Y$ be the
  analogue of $W$ where $H$ is replaced by $S$. Let
  \[ A_{ij}(u_0, \ldots, u_5) = (W^* Z Y^*)_{ij} + (W^* Z
    Y^*)_{ji}. \] Then the forms $\Xi$ and $M$ in Lemma~\ref{lem:XiM}
  satisfy
  \begin{equation}
  \label{rank1identity}
   -f_6 w(\Xi,M) = \alpha_\eps \left( \sum c_{ij} A_{ij} \right)^2
  \end{equation}
  where $\alpha_\eps$ and $Q_\eps = \sum c_{ij} x_i x_j$ are as given in
  Proposition~\ref{prop:quartic-factors}.
\end{Proposition}

\begin{proof}
  Since $P_\eps \in L_{10}[x_1, \ldots, x_4]$ is a quartic form, we may write
  \[ P_\eps(x_1, \ldots,x_4) = \widetilde{P}_{\eps}(x_1^2,x_1 x_2,
    \ldots, x_4^2)\] where $\widetilde{P}_{\eps}$ is a quadratic form
  in $10$ variables.  This does not determine $\widetilde{P}_{\eps}$
  uniquely, but in view of~\eqref{eqn:alpha} there is a canonical
  choice, namely the one for which the associated symmetric bilinear
  form sends
\begin{equation*}
  (x_1^2, x_1 x_2, \ldots, x_4^2 ; y_1^2, y_1 y_2, \ldots, y_4^2 )
  \mapsto \alpha_\eps Q_\eps(x_1, \ldots, x_4) Q_\eps(y_1, \ldots,y_4).
\end{equation*}

The proof of~\eqref{rank1identity} is now reduced to showing that
\begin{equation*}
-f_6 w(\Xi,M) = \widetilde{P}_{\eps}(A_{11},A_{12}, \ldots, A_{44}).
\end{equation*}
Each side is a form of degree $6$ in $u_0, \ldots, u_5$ whose
coefficients have degree $15$ in the coefficients $h_{ij}$ of the
model. We checked this by computer algebra for the model $(G,H)$ in
Section~\ref{sec:desing}.  The general case follows by carefully
keeping track of the effect of a change of coordinates.
\end{proof}

We solve for $s$ and $\mu$ in Theorem~\ref{thm:twisted222} as follows.
First we pick $(\xi_\eps,m_\eps)$ by specialising the forms $(\Xi,M)$
in Lemma~\ref{lem:XiM}.  Specialising~\eqref{rank1identity} at the
same point gives $-f_6 w(\xi_\eps,m_\eps) = \alpha_\eps \chi_\eps^2$ for
some explicit $\chi_\eps \in L_{10}$, which is a unit if we
specialised at a sufficiently general point. We repeat everything with
$\eps$ replaced by $\eta$ or $\eps+\eta$.  By~\eqref{xi-m-pairs} we
have
\begin{equation}
\label{sqrt}
 (\xi_\eps \xi_\eta \xi_{\eps+\eta}, m_\eps m_\eta m_{\eps+\eta})
 = (r \nu^2, r^3 N_{L/k}(\nu)),
\end{equation}
for some $r \in k^\times$ and $\nu \in L^\times$.  We then use the
next lemma to solve for $s$ and $\mu$ satisfying
$\alpha_\eps \alpha_\eta/ \alpha_{\eps+\eta} = s \mu^2$.

\begin{Lemma}
\label{lem:hom}
Let $(\xi_i,m_i) \in L^\times \times k^\times$ with $N_{L/k}(\xi_i) = m_i$
for $i=1,2,3$.
  If
  \begin{equation*}
    (\xi_1 \xi_2 \xi_3,m_1 m_2 m_3) = (r \nu^2, r^3 N_{L/k}(\nu))
      \end{equation*} then
  \[ \prod_{i=1}^3 w(\xi_i,m_i) = s \mu^2 \] for some
  $s \in k$ and $\mu \in L_{10}$. Moreover we may
  take $s = r^3$ and \begin{equation}
    \label{eqn:kappa} \mu(\gamma_1 | \gamma_2)
    = \left( \prod_{\theta \in \gamma_1} \nu(\theta) \right)
    \prod_{i=1}^3 \left(
      1 + \frac{m_i}{\prod_{\theta \in \gamma_1} \xi_i(\theta)} \right).
    \end{equation}
\end{Lemma}
\begin{proof}
  If the $w(\xi_i,m_i)$ are units, then the existence of $s$ and $\mu$
  follows from Lemma~\ref{lem:w} and the commutative
  diagram~\eqref{cd}. To prove the result in general and to give
  explicit formulae we proceed as follows.  By Definition~\ref{def:w}
  we have
  \[ w(\xi_i,m_i)(\gamma_1|\gamma_2) = \left( \prod_{\theta \in
        \gamma_1} \xi_i(\theta) \right) \left( 1 +
      \frac{m_i}{\prod_{\theta \in \gamma_1} \xi_i(\theta)}
    \right)^2. \] It follows that
  $\prod_{i=1}^3 w(\xi_i,m_i) = s \mu^2$ where $s$ and $\mu$ are as
  given in the statement of the lemma.  A calculation then shows
  that~\eqref{eqn:kappa} is unchanged when we swap
  $\gamma_1 \leftrightarrow \gamma_2$ and so defines an element of
  $L_{10}$.
\end{proof}

We have now reduced the problem of solving for $s$ and $\mu$ in
Theorem~\ref{thm:twisted222} to that of solving for $r$ and $\nu$
in~\eqref{sqrt}.  This still does not have a unique solution, but our
choices for $\nu$ are more limited than they were for $\mu$.  Indeed
$\nu$ is unique up to multiplying by elements of $k^\times$ (which do
not matter) and elements of
\[ \frac{\{ \nu \in L^\times | \nu^2 \in k^\times,
    N_{L/k}(\nu) = \nu^6 \}}{k^\times} \isom \J[2](k).\]
This leaves us with $\# \J[2](k)$ choices. Since this is the same as
the number of choices we had in defining the twisted group law (see
Remark~\ref{rem:choices}), and so in defining the twisted
$(2,2,2)$-form, all of these choices work.

\section{Implementation and Examples}
\label{sec:examples}

Let $\C$ be a genus $2$ curve defined over a number field $k$, and let
$\J$ be its Jacobian. We have written a program in Magma \cite{magma}
for computing the Cassels-Tate pairing
\[ \langle~,~\rangle_{\CT} : S^{(2)}(\J/k) \times S^{(2)}(\J/k) \to
  \F_2 \] in the case $k = \Q$. We have also worked out some first
examples when $k$ is a small quadratic field.  The main steps are as
follows.

\begin{enumerate}
\item We compute the fake $2$-Selmer group $S_{\rm fake}^{(2)}(\J/k)$
  using the Magma function {\tt TwoSelmerGroup}. This calls programs
  of Stoll (when $k = \Q$) and Bruin (when $k \not= \Q$).  Our
  insistence that $\C$ has an equation $y^2 = f(x)$ where $f$ has
  degree $6$ can slow these programs down, so in the presence of a
  rational Weierstrass point it may be better to work with a degree
  $5$ model and then convert the fake $2$-Selmer group elements back
  to a degree $6$ model.
\item We compute the $2$-Selmer group $S^{(2)}(\J/k)$ from the fake
  Selmer group, representing its elements as pairs $(\xi,m)$ where
  $N_{L/k}(\xi) = m^2$ as in \eqref{xi-m-pairs}. Although this is
  simply a matter of extracting square roots of the norms, we must be
  careful to keep track of the group structure.
\item \label{getmodels} We convert each $2$-Selmer group element
  $(\xi,m)$ to a model $(G,H)$ as described in
  Section~\ref{sec:models}. This involves solving for a
  $3$-dimensional isotropic subspace of a quadratic form in $6$
  variables. When $k = \Q$ we use the existing function {\tt
    IsotropicSubspace} in Magma \cite{magma}.  This uses techniques of
  minimisation and reduction based on a paper of Simon
  \cite{Simon}. When $k$ is a small quadratic field we find that
  similar ideas work reasonably well in practice.
\item \label{minred} We use the action of $\PGL_4(k)$ on the space of
  models to simplify our list of models. This is achieved by a
  combination of minimisation and reduction. The algorithm for
  minimisation will be described in joint work of the first author and
  Mengzhen (August) Liu~\cite{FL}, based on a summer project in 2022. See
  Remark~\ref{rem:redn} below for details of the method we used for
  reduction.
\item \label{start} We select elements $\eps, \eta \in S^{(2)}(\J/k)$
  and aim to compute $\langle \eps, \eta \rangle_{\CT}$. If the value
  of the pairing can be inferred from earlier values computed, or from
  the properties of the pairing listed at the start of
  Section~\ref{sec:CTP} (specifically, it is symmetric, bilinear and
  satisfies Lemma~\ref{lem:PS}), then there is no need for the
  following steps (other than for testing purposes).  The deficient
  places may be computed using the algorithm in \cite[Section
  7]{Stoll}, as implemented in the Magma function {\tt IsDeficient}.
\item Starting from our models representing $\eps$, $\eta$ and
  $\eps+\eta$ we use~\eqref{getF} to compute equations for the twisted
  Kummer surfaces $\K_\eps$, $\K_\eta$ and $\K_{\eps+\eta}$.  As
  described in Section~\ref{sec:tw-kum}, computing $2 \times 2$ minors
  instead of $3 \times 3$ minors (and scaling by $-f_6$) gives us
  pushout forms for $\K_\eps$ and $\K_\eta$, i.e., quartic forms that
  when set equal to a square define the double covers
  $\J_\eps \to \K_\eps$ and $\J_\eta \to \K_\eta$.  Since we need them
  in Step~(\ref{get:ag}) below, we also compute the covariants
  describing the covering maps $\K_\eps \to \K$, $\K_\eta \to \K$ and
  $\K_{\eps + \eta} \to \K$.
\item \label{search} We search for $k$-points on
  $\K_\eta \subset \PP^3$. When $k=\Q$ we use the function {\tt
    PointSearch} in Magma, slightly modified to allow for the fact
  that Kummer surfaces are always singular.  When $k \not= \Q$ we took
  an ad hoc approach (intersecting with random lines) which could
  doubtless be improved.
\item \label{get:ag} Assuming we found a rational point
  $P \in \K_\eta(k)$, we compute $a \in k^\times/(k^\times)^2$ by
  evaluating the pushout form describing $\J_\eta \to \K_\eta$ at this
  point $P$. We then compute the rational function $g$ on $\K_\eps$
  using Theorems~\ref{thm:getg} and~\ref{thm:twisted222}.  There is no
  need to compute the $(2,2,2)$-form in full, since we immediately
  specialise the second and third sets of variables.  These are
  specialised to the coordinates of $P \in \K_\eta \subset \PP^3$ and
  an arbitrary point on $(\PP^3)^\vee$. For the latter we take
  $(c_1,c_2,c_3,c_4) = (1,0,0,0)$.
\item \label{end} We compute $\langle \eps, \eta \rangle_{\CT}$ using
  Theorem~\ref{thm:ctp-formula}(iii). We are careful to pick local
  points on $\K_\eps$ that lift to local points on $\J_\eps$.  The sum
  over all places is reduced to a finite sum as explained in
  Section~\ref{sec:badp} below.
\item We repeat Steps~(\ref{start}) to~(\ref{end}) until we have
  computed the matrix of the Cassels-Tate pairing
  $\langle~,~\rangle_{\CT}$ relative to a basis for $S^{(2)}(\J/k)$ as
  an $\F_2$-vector space.
\end{enumerate}

\begin{Remark}
\label{rem:searching}
For our methods to work, we do not need to find rational points on all
of the twisted Kummer surfaces.  Accordingly, in Step~(\ref{search})
we use a relatively small search bound, and only return to search
again if we do not succeed with other choices of $\eta$.  For this
reason it is recommended to give our program the full $2$-Selmer group
as input, and not some subgroup or just a pair of elements.
\end{Remark}

\begin{Remark}
  The main obstacles to generalising to number fields of larger
  degree are the need to carry out minimisation and reduction in
  Steps~(\ref{getmodels}) and~(\ref{minred}), and the point search in
  Step~(\ref{search}).  We found it particularly helpful if the ring
  of integers of $k$ is a Euclidean domain since then our code for
  minimisation over $\Q$ (which calls the Magma function {\tt
    SmithForm}) could be adapted relatively easily.
\end{Remark}

\begin{Remark}
\label{rem:redn}
For simplicity we take $k=\Q$. In the reduction part of
Step~(\ref{minred}) we are looking for a matrix in $\GL_4(\Z)$
representing a change of coordinates that simplifies our equation for
the twisted Kummer surface $\K_\eps \subset \PP^3$.  The same matrix,
acting as described in Definition~\ref{def:model}(ii), may then be
used to simplify the model $(G,H)$.  To find this matrix in
$\GL_4(\Z)$ we use the natural generalisation
of~\cite[Corollary~6.3]{minred234}, namely we embed
$\Qbar \subset \CC$ and perform lattice reduction with respect to the
Gram matrix
\[ \sum_{T \in \J[2]} \frac{M_T^\dagger M_T\,\, }{ \,\, | \det M_T \,
    |^{1/2}} \] where the matrices $M_T$ were computed in
Remark~\ref{rem:MT}, and the superscript $\dagger$ denotes complex
conjugate transpose.  An alternative method would be to run Magma's
{\tt ReduceCluster} on the set of nodes of $\K_\eps$, but the need to
numerically solve for the nodes makes this method less efficient.

Both of these methods are specific to curves of genus $2$.  A method
that works for $2$-coverings of hyperelliptic Jacobians in general is
currently being developed by Jack Thorne \cite{thorne-red}.
\end{Remark}

Presented with the output of a long and complicated program, it is
natural to wonder whether the answers computed are correct. We used
the methods in Section~\ref{sec:recover} to check that the models
computed in Step~(\ref{getmodels}) do indeed correspond to the Selmer
elements claimed. For computing the pairing itself the following
checks are available. The first was the most important for catching
bugs.

\begin{itemize}
\item In applying Theorem~\ref{thm:ctp-formula}(iii) we may compute
  many $k_v$-points on $\K_\eps$ and check that each makes the same
  local contribution to the pairing.
\item In Step~(\ref{get:ag}) we may make different choices of $P$, of
  $(c_1, \ldots,c_4)$, and of the overall scaling of $g$, and check
  that the pairing does not change.
\item We may check that the pairing as computed is symmetric, bilinear
  and satisfies Lemma~\ref{lem:PS}.
\item We may check that the image of the map
  $\J(k)/2\J(k) \to S^{(2)}(\J/k)$ is contained in the kernel of the
  pairing.
\item We may check that our answers are consistent with
  Remark~\ref{rem:creutz}.
\end{itemize}

In Section~\ref{sec:badp} we explain how the sum over all places in
our formula for the pairing reduces to a finite sum. In
Section~\ref{sec:first-example} we give some examples over $k=\Q$
where our methods improve the upper bound for the rank of $\J(\Q)$
coming from $2$-descent by $2$ or more. In Section~\ref{sec:bs} we
explain how our methods compare to the visibility method employed by
Bruin and Stoll, and how the methods can be combined. Finally in
Section~\ref{sec:lmfdb} we explain how we used our methods to
unconditionally determine the rank of every genus $2$ Jacobian in the
LMFDB.

\subsection{Controlling the bad primes}
\label{sec:badp}
We work over a number field $k$ with ring of integers $\Ok$.  Let $\C$
be a genus $2$ curve with equation
$y^2= f(x) = f_6 x^6 + f_5 x^5 + \ldots + f_1 x + f_0$ where the $f_i$
are in $\Ok$.  Let $S_0$ be a finite set of places of $k$ containing
all the infinite places and all the primes dividing
$2 f_6 \, {\rm disc}(f)$.
\begin{Theorem}
  \label{thm:badp}
  Let $v \not\in S_0$ be a place of $k$. Suppose
  that
  \begin{enumerate}
  \item $\K_\eps$ arises from a model $(G,H)$ (see
    Definition~\ref{def:model}) and $H$ has $v$-adically integral
    coefficients.
  \item $a \in k^\times$ is a $v$-adic unit.
  \item $g = \gamma(x_1, \ldots,x_4)/x_1^2$ where $\gamma \in k[x_1, \ldots,
    x_4]$ is a quadratic form, with all coefficients $v$-adically integral
    and at least one a $v$-adic unit.
  \end{enumerate}
  Then the local contribution at $v$ in Theorem~\ref{thm:ctp-formula}(iii)
  is trivial.
\end{Theorem}

\begin{proof}
  Replacing the local field $k_v$ by an unramified extension of odd
  degree we may assume that the order of the residue field ($q$ say)
  is arbitrarily large. Under the assumptions in the statement of the
  theorem, the construction of $\K_\eps \subset \PP^3$ and the double
  cover $\J_\eps \to \K_\eps$ from the model $(G,H)$ carries over to
  the residue field.  By the Lang-Weil estimates there are
  $\Omega(q^2)$ smooth points on the reduction of $\K_\eps$ mod $v$
  that lift to $\F_q$-points on the reduction of $\J_\eps$ mod $v$.
  The reduction of $\gamma$ mod $v$ cannot be identically zero on the
  surface (it is a quadratic form, whereas the surface has degree $4$)
  and so vanishes at $O(q)$ of these points. Since $q$ is large there
  are some points left over. Hensel lifting any one of these (and
  noting that for $v \nmid 2\infty$ the Hilbert symbol of two $v$-adic
  units is trivial) shows that the local contribution at $v$ is
  trivial.
\end{proof}

\subsection{First examples over $\Q$}
\label{sec:first-example}
Our first example was chosen as a genus $2$ curve with equation
$y^2 = f(x)$ where $f$ has small integer coefficients and Galois group
$S_6$, and a $2$-descent does not appear to give a sharp upper bound
for the rank of the Jacobian.

\begin{Example} \label{ex1}
Let $\C$ be the genus $2$ curve with equation
\[ y^2 = -3 x^6 + 3 x - 15. \] Let
$L = \Q[\theta] = \Q[x]/(x^6 - x + 5)$.  We represent elements of the
$2$-Selmer group $S^{(2)}(\J/\Q)$ as pairs
$(\xi,m) \in L^\times \times \Q^\times$ where $N_{L/\Q}(\xi) = m^2$.
Magma computes that $S^{(2)}(\J/\Q) \isom (\Z/2\Z)^3$ is generated by
\begin{align*}
c &=(1,-1), \\
\eps &= (\theta^5 - 3 \theta^4 - 2 \theta^3 + 5 \theta^2 + 4 \theta - 9, 1),\\
\eta &= (\theta^5 - 2 \theta^4 + 2 \theta^3 - \theta^2 + 1, 1).
\end{align*}
We find that the Selmer group elements $\eps$, $\eta$, $\eps + \eta$
are represented by models $(G,\frac{1}{2}H_1)$, $(G,\frac{1}{2}H_2)$,
$(G,\frac{1}{2}H_3)$ where $G = u_0 u_5 + u_1 u_4 + u_2 u_3$ and
\begin{align*}
  H_1 &= u_0 u_1 - 2 u_0 u_2 - u_1 u_2 + u_1 u_3 - 2 u_1 u_5 - 2 u_2 u_5
        - 2 u_3 u_4 - 2 u_4 u_5 + 2 u_5^2, \\
  H_2 &= u_0 u_2 - 2 u_0 u_3 - 2 u_1^2 - u_1 u_2 + u_2^2 - u_2 u_4
        + 3 u_2 u_5 + 2 u_3^2 + 2 u_3 u_5 - 4 u_4 u_5, \\
  H_3 &= -2 u_0 u_3 + u_1 u_2 + u_1 u_4 - 2 u_1 u_5 - u_2 u_3 + 2 u_2 u_4
        - u_3 u_4 + 4 u_4 u_5 + 2 u_5^2. 
\end{align*}
The identity element is represented by $(G,H)$ as given in
Section~\ref{sec:desing}, and the remaining four elements of
$S^{(2)}(\J/\Q)$ are given by reversing the order of the variables
$u_0, u_1, \ldots, u_5$.  (As noted in Remark~\ref{addc} this
corresponds to adding $c$.)  It may be checked using the formulae in
Section~\ref{sec:recover} that these models do indeed correspond to
the pairs $(\xi,m)$ we started with.

The procedure in Section~\ref{sec:tw-kum} shows that
$\K_\eps \subset \PP^3$ has equation $F_\eps = 0$ where
\begin{align*}
  F_{\eps} & = \, 2 x_1^3 x_2 + 4 x_1^3 x_3 + 4 x_1^3 x_4 - 18 x_1^2 x_2^2 +
  22 x_1^2 x_2 x_3 + 2 x_1^2 x_2 x_4 - 8 x_1^2 x_3^2 \\ & - 4 x_1^2 x_3 x_4 -
  2 x_1^2 x_4^2 + 6 x_1 x_2^3 + 6 x_1 x_2^2 x_3 + 7 x_1 x_2^2 x_4 - 6 x_1 x_2 x_3^2
  \\ & - 24 x_1 x_2 x_3 x_4 - 6 x_1 x_2 x_4^2 + 4 x_1 x_3^3 + 6 x_1 x_4^3 - x_2^4 +
  3 x_2^3 x_3 + 3 x_2^3 x_4 \\ & - 3 x_2^2 x_3 x_4 - 5 x_2^2 x_4^2 + 2 x_2 x_3^3 -
  2 x_2 x_3^2 x_4 + 2 x_2 x_4^3 - 2 x_3^2 x_4^2 + 2 x_3 x_4^3 + x_4^4.
\end{align*}
We likewise find equations for $\K_\eta \subset \PP^3$ and
$\K_{\eps + \eta} \subset \PP^3$. On $\K_\eta$ we find the rational
point $P = (1 : 0 : -1 : -1)$ whose inverse image in $\J_\eta$ is
defined over $\Q(\sqrt{-3})$. Since $\Gal(f) = S_6$, computing the
twisted $(2,2,2)$-form is simplified as described in
Remark~\ref{rem:S6}.  Taking $(c_1,c_2,c_3,c_4)=(1,0,0,0)$ in
Theorem~\ref{thm:getg} we find that $g = \gamma(x_1,\ldots,x_4)/x_1^2$
where $\gamma = 12 x_1 x_2 - 2 x_2^2 + 6 x_2 x_3 + 3 x_2 x_4$.
According to Theorem~\ref{thm:ctp-formula} the Cassels-Tate pairing is
given by
\[ \langle \eps,\eta \rangle_\CT = \sum_v (-3,\gamma(P_v))_v \] where
for each place $v$ of $\Q$ we choose a local point
$P_v \in \K_\eps(\Q_v)$ that lifts to $\J_\eps(\Q_v)$ and satisfies
$\gamma(P_v) \not= 0$. As before, our convention is that the Hilbert
symbol $(~,~)_v$ takes values in $\F_2$.

By Theorem~\ref{thm:badp} the only places that could contribute to the
pairing are the primes dividing the discriminant
$-2^8 \cdot 3^{10} \cdot 5^6 \cdot 7 \cdot 31 \cdot 43$, together with
$\infty$.  For each of these places $v$ we specify a local point
$P_v \in \K_\eps(\Q_v)$ whose first three coordinates are exact, and
whose last coordinate is either exact or given to sufficient precision
to determine the point uniquely.  It is important to note that we have
chosen local points that lift to $\J_\eps(\Q_v)$. The third column
lists $\gamma(P_v)$ as an element of $\Q_v^\times/(\Q_v^\times)^2$.
\[ \begin{array}{c|c|c|c}
     \text{ place } v & \text{ local point $P_v$ on $\K_\eps$ } &
    \gamma(P_v) & (-3,\gamma(P_v))_v \\ \hline
2 & (3 : 4 : 1 : 2^2 + O(2^3)) & -2 & 1 \\
3 & (1 : -1 : 0 : -3^2 + O(3^3)) & 1 & 0 \\
5 & (0 : 1 : 1 : 1) &    2 & 0 \\
7 & (0 : 1 : 0 : 1) &    1 & 0 \\
31 & (0 : 1 : 0 : 1) &   1 & 0 \\
43 & (0 : 1 : 0 : 1) &   1 & 0 \\
\infty &  (0 : 1 : 1 : 1) &   1 & 0
\end{array} \]
Adding up the entries in the right hand column shows that
$\langle \eps,\eta \rangle_{\CT} \not= 0$.

Since $\C$ has an even number of deficient places (these are $3$ and
$\infty$) we know by Lemma~\ref{lem:PS} that the Cassels-Tate pairing
on $S^{(2)}(\J/\Q)$ has even rank. Since we have shown that this
pairing is non-zero, this improves the upper bound
$\rank \, \J(\Q) \leqslant 3$ coming from $2$-descent to
$\rank \, \J(\Q) \leqslant 1$.

Further calculations show that the pairing is given by
\[ \begin{array}{c|ccc}
\langle ~,~ \rangle_{\CT} & c & \eps & \eta \\ \hline
c &     0 & 1 & 1 \\
\eps &  1 & 1 & 1 \\
  \eta &  1 & 1 & 1
\end{array} \]

Since the kernel of the pairing is $1$-dimensional, spanned by
$\eps + \eta$, this suggests we should look for rational points on
$\J_{\eps+\eta}$.  We find that the point $(1:-2:-2:0)$ on
$\K_{\eps+\eta}$ lifts to a $\Q$-point on $\J_{\eps+\eta}$. Using the
formulae for the covering map in Section~\ref{sec:tw-kum}, this point
maps to the point $( 124 : 238 : 199 : 3607 )$ on $\K$, which in turn
lifts to a point of infinite order in $\J(\Q)$, represented by a pair
of points on $\C$ with $x$-coordinates satisfying
$124 x^2 - 238 x + 199 = 0$. It follows that $\rank \, \J(\Q) = 1$ and
$\Sha(\J/\Q)[2] \isom (\Z/2\Z)^2$.  Our calculations further show that
the Cassels-Tate pairing on $\Sha(\J/\Q)[2]$ is non-degenerate, and
hence $\Sha(\J/\Q)[2^{\infty}] \isom (\Z/2\Z)^2$.
\end{Example}

Next we give an example where the rank of $\J(\Q)$ had previously only
been computed conditional on the Birch Swinnerton-Dyer conjecture. We
discuss further such examples in Sections~\ref{sec:bs}
and~\ref{sec:lmfdb}.

\begin{Example}
\label{ex:firstLMFDB}
Let $\C$ be the genus $2$ curve with equation
\[ y^2 = -3 x^6 - 4 x^5 - 10 x^4 - 51 x^2 + 80x - 28, \] and LMFDB
label {\tt 49471.a.49471.1}. Since the right hand side factors as the
product of two cubics, the canonical element of the $2$-Selmer group
is trivial. It follows by Lemma~\ref{lem:PS} that the Cassels-Tate
pairing on $S^{(2)}(\J/\Q) \isom (\Z/2\Z)^2$ is alternating. A
calculation similar to that in the previous example shows that the
pairing is non-zero.  Therefore $\rank \, \J(\Q) = 0$ and
$\Sha(\J/\Q)[2^{\infty}] \isom (\Z/2\Z)^2$.
\end{Example}

Our next example is taken from a paper of Logan and van Luijk
\cite{Logan-vanLuijk}.

\begin{Example} Let $\C$ be the genus $2$ curve with equation
  \[ y^2 = -6(x^2 + 1)(x^2 - 2 x - 1)(x^2 + x - 1).\] We have
  $\J(\Q)[2] \isom (\Z/2\Z)^2$.  We find that
  $S^{(2)}(\J/\Q) \isom (\Z/2\Z)^4$ is generated by
\begin{align*}
    c &= (1, -1), \\
    \eps &= (\theta^4 - \theta^3 - \theta^2 - 2 \theta - 2, 18), \\
    \eta &= (5 \theta^3 - 14 \theta^2 + 1, 900), \\
    \nu &= (\theta^4 + \theta^3 + 5 \theta^2 - 3, 270),
\end{align*}
and the Cassels-Tate pairing on $S^{(2)}(\J/\Q)$ is given by
\[ \begin{array}{c|cccc}
\langle ~,~ \rangle_{\CT} & c & \eps & \eta & \nu \\ \hline
c &      0 & 1 & 0 & 0 \\
\eps & 1 & 1 & 0 & 0 \\
\eta & 0 & 0 & 0 & 0 \\
\nu & 0 & 0 & 0 & 0
\end{array} \]
It follows that $\rank \, \J(\Q) = 0$
and $\Sha(\J/\Q)[2^{\infty}] \isom (\Z/2\Z)^2$.

We were able to compute the pairing in this example despite the result
of Logan and van Luijk \cite[Proposition~3.33]{Logan-vanLuijk} that
$\K_{\eps}(\Q) = \emptyset$.  This illustrates the point we made in
Remark~\ref{rem:searching}.  In fact, in this example we find that of
the $15$ twisted Kummer surfaces (associated to the non-zero elements
of $S^{(2)}(\J/\Q)$) all except $\K_{\eps}$ and
$\K_{\eps + c} \isom (\K_\eps)^\vee$ have rational points of small
height.
\end{Example}

We end this section with an example where the pairing has higher rank.

\begin{Example}
Let $\C$ be the genus $2$ curve with equation
\[ y^2 = f(x) = 3 x^6 + 10 x^5 - 16 x^4 + 18 x^3 + 23 x^2 - 54 x -
  17. \] We have $\Gal(f) = S_6$ and so the Jacobian $\J$ has no
rational $2$-torsion.  We find that $S^{(2)}(\J/\Q) \isom (\Z/2\Z)^5$
is generated by
\begin{align*}
  c &= (1, -1), \\
  \eps &=   (3 \theta^2 - 8 \theta + 10, 2041), \\
  \eta &=   (2 \theta^2 - 4 \theta + 15, 16657/3), \\
  \nu &=    (6 \theta^3 + 14 \theta^2 - 16 \theta - 31, 37503), \\
  \phi & =  (2 \theta^4 + 9 \theta^3 - 3 \theta^2 - 5 \theta + 1, 2925),
\end{align*}
and the Cassels-Tate pairing on $S^{(2)}(\J/\Q)$ is given by
the rank $3$ matrix
\[ \begin{array}{c|ccccc}
\langle ~,~ \rangle_{\CT} & c & \eps & \eta & \nu & \phi \\ \hline
c & 1 & 0 & 0 & 0 & 1 \\
\eps & 0 & 0 & 1 & 1 & 0 \\
\eta & 0 & 1 & 0 & 0 & 1 \\
\nu & 0 & 1 & 0 & 0 & 1 \\
\phi & 1 & 0 & 1 & 1 & 1 \\
   \end{array} \]
 The points $( 11 : -98 : 78 : -2217 )$ and $( 44 : 72 : 124 : -245 )$
 on the Kummer surface $\K$ lift to independent points of infinite
 order in $\J(\Q)$. It follows that $\rank \, \J(\Q) = 2$ and
 $\Sha(\J/\Q)[2^\infty] \isom (\Z/2\Z)^3$.
\end{Example}

\subsection{Examples from an experiment of Bruin and Stoll}
\label{sec:bs}
Bruin and Stoll conducted an experiment~\cite{BruinStoll} where they
attempted to determine the existence of rational points on all genus
$2$ curves $y^2 = f_6 x^6 + f_5 x^5 + \ldots + f_1 x + f_0$ where the
$f_i$ are integers with $|f_i| \leqslant 3$.  They reduced this to the
determination of the ranks of $47$ genus $2$ Jacobians, of which $36$
are expected to have rank $0$, a further $10$ are expected to have
rank $1$, and one is expected to have rank $2$.

We start by considering one of the cases where the rank is expected to
be $1$.
\begin{Example}
\label{ex-BS}
Let $\C$ be the genus $2$ curve with equation
\[ y^2 = -x^6 + 2x^5 + 3x^4 + 2x^3 - x - 3. \]
Magma computes that $S^{(2)}(\J/\Q) \isom (\Z/2\Z)^3$ is generated by
\begin{align*}
  c  &= (1, -1), \\
\eps &= (\theta^2 + \theta + 1, 4), \\
\eta &= (\theta^3 - 4 \theta^2 + 3 \theta - 1, 8).
\end{align*}
We find that the Selmer group elements $\eps$, $\eta$, $\eps+\eta$ are
represented by the models $(G,\frac{1}{2}H_1)$, $(G,\frac{1}{2}H_2)$,
$(G,\frac{1}{2}H_3)$, where $G = u_0 u_5 + u_1 u_4 + u_2 u_3$ and
\begin{align*}
  H_1 &=  u_0 u_2 - u_0 u_4 + 2 u_1 u_3 - 2 u_1 u_5
        + 2 u_2^2 + 2 u_2 u_3 + 4 u_3^2 - 10 u_3 u_5 - 2 u_4 u_5, \\
  H_2 &=  2 u_0^2 + 2 u_0 u_1 + 2 u_0 u_4 + 4 u_1 u_5
        + 2 u_2^2 + 2 u_2 u_3 + 4 u_2 u_5 - u_3 u_4 + 3 u_3 u_5, \\
  H_3 &=  u_0 u_4 + 2 u_1 u_2 - 2 u_1 u_3 + 4 u_1 u_4 - 2 u_2 u_3
        + 8 u_3 u_5 - u_4^2 - 6 u_4 u_5 - 2 u_5^2.
\end{align*}
Using these models we compute that the Cassels-Tate pairing is given by
\[ \begin{array}{c|ccc}
\langle ~,~ \rangle_{\CT} & c & \eps & \eta \\ \hline
   c  &  0  &  0  &  0 \\
 \eps &  0  &  0  &  1 \\
 \eta &  0  &  1  &  0
\end{array} \]
Since the point $( 9536 : -5312 : 5008 : 53113 )$ on $\K$ lifts to a point
of infinite order in $\J(\Q)$ it follows that $\rank \, \J(\Q) = 1$.
\end{Example}

\begin{Remark}
  \label{rem:vis}
  Let $d$ be a squarefree integer.  We write $\C_d$ for the quadratic
  twist of $\C$ by $d$, and $\J_d$ for its Jacobian. It is well known
  that
\begin{equation}
  \label{eqn:add}
  \rank \, \J(\Q(\sqrt{d})) = \rank \, \J(\Q) + \rank \, \J_d(\Q).
\end{equation}
Bruin (see \cite{BruinFlynn}, \cite{BruinStoll}) observed that the
upper bound for the rank of $\J(\Q)$ coming from $2$-descent can
sometimes be improved by carrying out a $2$-descent over
$\Q(\sqrt{d})$. Indeed, this gives an upper bound for the rank of
$\J(\Q(\sqrt{d}))$ which, when combined with~\eqref{eqn:add} and a
lower bound for the rank of $\J_d(\Q)$, gives an upper bound for the
rank of $\J(\Q)$. When this method succeeds it is because there are
elements in $\Sha(\J/\Q)[2]$ that are in the kernel of the restriction
map \[\Sha(\J/\Q) \to \Sha(\J/\Q(\sqrt{d})).\] One way to think about
this kernel is as the elements of $\Sha(\J/\Q)$ that are visible (in
the sense defined by Mazur) in the restriction of scalars of
$\J$. Accordingly Bruin calls his method {\em visibility}.
\end{Remark}

Bruin and Stoll \cite{BruinStoll} were already able to show that
$\rank \, \J(\Q) = 1$ in Example~\ref{ex-BS} by taking $d = -1$ in
Remark~\ref{rem:vis}. In fact this is just one of $5$ curves on their
list for which the same thing happens, that is, by computing the
Cassels-Tate pairing we are able to prove that $\rank \, \J(\Q) = 1$,
but Bruin and Stoll already proved this using visibility. Likewise the
conclusion $\rank \, \J(\Q) = 1$ in Example~\ref{ex1} can be proved by
taking $d = -3$ in Remark~\ref{rem:vis}.  However our method has the
advantage that we only had to carry out a $2$-descent over $\Q$
instead of $\Q(\sqrt{d})$, and so only had to verify a class group
calculation in a degree 6 rather than a degree 12 number field.

In the remaining $42$ examples we find that the Cassels-Tate pairing
on $S^{(2)}(\J/\Q)$ is identically zero. This is consistent with the
calculations of Bruin and Stoll, who showed in each case that
$\Sha(\J/\Q)$ has analytic order (approximately) $16$, and so is
expected to be isomorphic to $(\Z/4\Z)^2$. Visibility was not
sufficient for Bruin and Stoll to compute the ranks in these cases,
except in $4$ cases where they obtained results conditional on GRH.
However by combining visibility with computing the Cassels-Tate
pairing over a quadratic field $\Q(\sqrt{d})$, we are able to prove
that in all but $4$ of the $42$ examples the Jacobian does have the
expected rank. For this we took $d \in \{-1,\pm 2,\pm
3,5,-7\}$. In the unresolved cases the expected rank is~$0$, and
the problem is the apparent lack of a suitable small $d$.

We give further details for the final curve on their list.

\begin{Example}
Let $\C$ be the genus $2$ curve
\[ y^2 = f(x) = 3 x^6 + 3 x^5 + x^4 - 3 x^3 - 3 x^2 + x - 3. \] Since
$\Gal(f) = S_6$ there are no $2$-torsion points on $\J = \Jac(\C)$,
even over a quadratic extension.  It is easy to find two independent
$\Q$-points of infinite order on both $\J$ and its quadratic twist by
$-3$.  Let $k = \Q(\zeta)$ where $\zeta = (-1 + \sqrt{-3})/2$.  Magma
computes\footnote{This takes a few seconds assuming GRH, and still
  less than 10 minutes without.}  that
$S^{(2)}(\J/k) \isom (\Z/2\Z)^6$ is generated by
\begin{align*}
  \alpha_1 &= (1,-1), \\
  \alpha_2 &= (\theta^2 + \theta,1), \\
  \alpha_3 &= (3 \theta^4 - 3 \theta^2 + 1,169/3), \\
  \alpha_4 &= (3 \theta^5 + 3 \theta^4 + \theta^3 + 3 \theta^2 + 1,9), \\
  \eps &= (3(1 + \zeta) \theta^4 - 10 \theta^2
            - 8 \zeta \theta - 5 - 2 \zeta,  -1543 - 2462 \zeta), \\
  \eta &= (3(2 + \zeta) \theta^3 - 2 (2 + \zeta) \theta, -171),
\end{align*}
where as usual $\theta$ is a root of $f$. In fact
$S^{(2)}(\J/\Q) \isom (\Z/2\Z)^4$ is generated by
$\alpha_1, \ldots, \alpha_4$, but the Cassels-Tate pairing on this
Selmer group is trivial.  We represent the Selmer group elements
$\eps$, $\eta$, $\eps + \eta$ by the models $(G,\lambda^{-1} H_1)$,
$(G,\lambda^{-1} H_2)$, $(G,\lambda^{-1} H_3)$ where
$G = u_0 u_5 + u_1 u_4 + u_2 u_3$, $\lambda = 4(1-\zeta)$ and
\begin{align*}
  H_1 &= (2 + \zeta) u_0^2 - 2 (1 - \zeta) u_0 u_4
    - 2 (1 - \zeta) u_0 u_5 + 2 u_1^2 - 4 u_1 u_3 + 5 (1 + 2 \zeta) u_2^2 \\
  & + 6 (3 + 2 \zeta) u_2 u_4 + 4 (7 + 8 \zeta) u_2 u_5 - 2 u_3^2 + 8 u_3 u_4
    + 2 u_4^2 + 18 u_4 u_5 + (29 + 16 \zeta) u_5^2, \\
  H_2 &= 2 u_0 u_2 + 2 (1 - \zeta) u_0 u_3 - \zeta u_1^2
    - 2 (1 - \zeta) u_1 u_4 - 4 u_2^2 + 4 u_2 u_4 + 8 u_2 u_5 \\
  & - 4 (1 - \zeta) u_3 u_4 - 8 (1 - \zeta) u_3 u_5 - u_4^2
  + 8 (1 + 3 \zeta) u_4 u_5 - 16 u_5^2, \\
  H_3 &= 4 u_0^2 + 8 u_0 u_1 + 4 u_0 u_2 - 4 \zeta u_0 u_3 - 8 \zeta u_0 u_4
  + (3 - \zeta) u_1^2 + 4 (1 + \zeta) u_1 u_2 \\
  & - 2 (1 - \zeta) u_1 u_4 + 2 \zeta u_2^2 + 2 (1 + 2 \zeta) u_2 u_5
  - 2 (1 - 2 \zeta) u_3^2 - 4 u_3 u_4 + 2 (2 + \zeta) u_3 u_5 \\
  &  - (1 - 3 \zeta) u_4^2.
\end{align*}
On $\K_\eta$ we find the point
$( \zeta : 1 : 2 + 2 \zeta : 2 - 4 \zeta )$.  This lifts to a point on
$\J_\eta$ defined over $k(\sqrt{a})$ where $a = 7 + 4 \zeta$. Using
Theorem~\ref{thm:getg} with $(c_1, \ldots,c_4) = (1,0,0,0)$ we compute
that $g = \gamma(x_1, \ldots,x_4)/x_1^2$ where
\begin{align*}
  \gamma &= -(233 + 446 \zeta) x_1^2 - (276 - 308 \zeta) x_1 x_2
   + ( 178 - 316 \zeta) x_1 x_3 \\ & + ( 118 + 260 \zeta) x_1 x_4
   + (228 - 196 \zeta) x_2^2 - (4 + 188 \zeta) x_2 x_3
\\ &  + (84 - 300 \zeta) x_2 x_4 + (227 + 34 \zeta) x_3^2 - (6 - 36 \zeta)
      x_3 x_4 - ( 17 + 82 \zeta) x_4^2.
\end{align*}
According to Theorem~\ref{thm:ctp-formula} the Cassels-Tate pairing is
given by
\[ \langle \eps,\eta \rangle_\CT = \sum_v (a,\gamma(P_v))_v \] where
for each place $v$ of $k$ we choose a local point
$P_v \in \K_\eps(k_v)$ that lifts to $\J_\eps(k_v)$ and satisfies
$\gamma(P_v) \not= 0$.

Since $\C$ has discriminant
$2^8 \cdot 3^4 \cdot 13 \cdot 17 \cdot 89 \cdot 6229$, the norm of $a$
is $37$, the coefficients of $\gamma$ generate the unit ideal in
$\Z[\zeta]$, and $k$ has no real places, it follows by
Theorem~\ref{thm:badp} that the only places that could contribute to
the pairing are the primes dividing $2,3,13,17,37,89$ and $6229$.  For
each of these primes $\pp$, except $(2)$, it is easy to find a smooth
point on the reduction of $\K_{\eps}$ mod $\pp$ where the pushout form
defining the double cover $\J_{\eps} \to \K_{\eps}$ takes a non-zero
square value, and $\gamma$ takes a non-zero value. Moreover if $\pp$
is the prime dividing $a$, then $\gamma$ takes a non-zero square
value.  Taking as our local point any Hensel lift of this mod $\pp$
point shows that these primes make no contribution to the pairing.  It
remains to compute the contribution at the prime $(2)$.  In this case
we choose the local point
\[ P_2 = ( 1 : 0 : 1 : 2^2 + (1 + \zeta) 2^3 + O(2^5)) \in
  \K_\eps(k_2) \] which lifts to $\J_\eps(k_2)$, and satisfies
$\gamma(P_2) \equiv 2^2 (-1 + 2 \zeta) \pmod{2^5}$. Since the Hilbert
symbol $(a,-1 + 2 \zeta)_2$ is non-trivial it follows that
$\langle \eps,\eta \rangle_\CT \not = 0$.

Since $\C$ has no deficient places, we know by Lemma~\ref{lem:PS} that
the Cassels-Tate pairing on $S^{(2)}(\J/k)$ has even rank.  In view of
the points of infinite order we found earlier, and~\eqref{eqn:add}, it
follows that $\rank \, \J(k) = 4$ and $\rank \, \J(\Q) = 2$.
\end{Example}

\subsection{Examples from the LMFDB}
\label{sec:lmfdb}
The {\em $L$-functions and modular forms database} \cite{lmfdb}
contains a database of genus $2$ curves defined over $\Q$, as
described in the accompanying paper~\cite{lmfdb-paper}. The database
(accessed May 2023) contains $66,158$ genus $2$ curves in $65,534$
isogeny classes, and contains all known genus $2$ curves with absolute
discriminant at most $10^6$.

Amongst the data listed (for each genus $2$ curve $\C$ with Jacobian
$\J$) is the analytic order of the Tate--Shafarevich group
$\Sha(\J/\Q)$.  We restrict attention to the $4161$ curves where this
order is even.  In all $4161$ cases we were able to compute the
Cassels-Tate pairing on $S^{(2)}(\J/\Q)$. In particular the fact that
our method depends on finding rational points on certain twisted
Kummer surfaces was not a problem in practice.

In $4088$ of the $4161$ cases we find that the Cassels-Tate pairing on
$\Sha(\J/\Q)[2]$ is non-degenerate, and hence a $2$-descent followed
by our methods are sufficient to compute the rank of $\J(\Q)$
unconditionally. In such examples it also follows that
$\Sha(\J/\Q)[2^\infty]$ is $2$-torsion.  These examples, broken down
by $r = \rank \, \J(\Q)$ and $s= \dim_{\F_2} \Sha(\J/\Q)[2]$, were
distributed as follows.
\[ \begin{array}{c|cccc|c}
  &s=1&s=2&s=3&s=4& {\rm total}\\ \hline
r=0&   1125 & 1387  & 38 &   9  &  2559 \\
r=1&   1004 &  406  &  7  &  2  &  1419 \\
r=2&   106  &   3  &  0 &   0  &  109 \\
r=3&      1  &  0  &  0  &  0  &  1 \\ \hline
  {\rm total} &  2236 & 1796 & 45 & 11 & 4088
   \end{array} \]
 We recall (see Lemma~\ref{lem:PS}) that a place $v$ is deficient
 if $\C$ has no $\Q_v$-rational divisor of degree~$1$. The same examples,
 broken down by the number $t$ of deficient places and
 $s= \dim_{\F_2} \Sha(\J/\Q)[2]$, were distributed as follows.
\[ \begin{array}{c|cccc|c}
  &s=1&s=2&s=3&s=4& {\rm total}\\ \hline
t = 0 &  0   & 1782 &   0 &  10 & 1792 \\
t = 1 & 2232 &    0 &  45 &   0 & 2277 \\
t = 2 & 0    &   14 &   0 &   1 & 15 \\
t = 3 & 4    &    0 &   0 &   0 & 4 \\ \hline
 {\rm total} & 2236 & 1796 & 45 & 11 & 4088  
\end{array} \] 
This is in accordance with the result of Poonen and Stoll (see
Lemma~\ref{lem:PS}) that the rank of the pairing has the same parity
as the number of deficient places. The pairing was alternating
(equivalently, the canonical element $c$ was in the kernel of the
pairing) in $486$ cases with $s=2$, and $3$ cases with $s=4$.

\medskip


We now compute the rank of $\J(\Q)$ in the remaining $4161 -4088 = 73$
cases.  We divide these curves into lists $A$, $B$, $C$ and $D$
according to the method we use.  These lists contain $17$, $30$, $15$
and $11$ curves respectively.

The Jacobians in list $A$ are isogenous over $\Q$ to a product of
elliptic curves.  Since the elliptic curves all have small conductor
(at most $25840$), it is easy to compute the rank of $\J(\Q)$ (it is
either $0$ or $1$) as the sum of the ranks of the elliptic curves.  In
contrast the Jacobians in lists $B$, $C$ and $D$ are absolutely
simple. In these cases it is expected that $\rank \, \J(\Q) = 0$.  The
Birch Swinnerton-Dyer conjecture (together with a $2$-descent and
Lemma~\ref{lem:PS}) further predicts that
$\Sha(\J/\Q) \isom (\Z/4\Z)^2$ or $(\Z/8\Z)^2$, the latter only for
the curve with LMFDB label {\tt 108737.a.108737.1} in list $C$.

We say that genus $2$ curves are isogenous if their Jacobians are
isogenous. For each curve in list $B$ we found an isogenous curve
whose Jacobian has no elements of order $2$ in its Tate-Shafarevich
group.  A $2$-descent was therefore sufficient to prove that the rank
is $0$.  For each curve in list $C$ we likewise found an isogenous
curve whose Jacobian has no elements of order $4$ in its
Tate-Shafarevich group.  Computing the Cassels-Tate pairing on the
$2$-Selmer group was therefore sufficient to prove that the rank is
$0$.

In processing lists $B$ and $C$, we found the isogenous curves using
the Magma function {\tt TwoPowerIsogenies}.  In all but two cases
({\tt 111989.a.111989.2} and {\tt 276083.a.276083.2}, both in list
$C$) the isogenous curve we used is {\em not} listed in the LMFDB,
since it has absolute discriminant greater than $10^6$.  The isogenies
that are relevant here are typically ``double Richelot
isogenies''. These are isogenies defined over $\Q$ that are the
composite of two Richelot isogenies defined over a cubic number field.

There are $11$ curves in list $D$, in $10$ isogeny classes. We were
able to prove that the rank is $0$ in each of these cases using our
method from Section~\ref{sec:bs}, that is, we combine visibility with
computing the Cassels-Tate pairing over a quadratic field
$\Q(\sqrt{d})$.  We used the following values of $d$, ignoring the
first curve since it is isogenous to the second (in fact via a double
Richelot isogeny).
\[\begin{tabular}{cc|cc|cc}
    LMFDB label & $d$ & LMFDB label & $d$ & LMFDB label & $d$ \\ \hline
 {\tt 65563.d.65563.1}   &&  {\tt 300429.a.300429.1} &$-3$& {\tt 876096.a.876096.1} &$-3$\\
 {\tt 65563.d.65563.2}   &$-7$&  {\tt 438837.b.438837.1} &$-3$& {\tt 881669.a.881669.1} &$5$\\
 {\tt 106015.b.742105.1} &$-3$&  {\tt 681797.a.681797.1} &$-3$& {\tt 913071.a.913071.1} &$-3$\\
 {\tt 270438.a.270438.1} &$-3$& {\tt 711917.a.711917.1}  &$-3$&
  \end{tabular}\]

This completes the unconditional determination of the ranks of all
genus $2$ Jacobians in the LMFDB.

\begin{Remark}
  Prior to our work, the LMFDB
  listed\footnote{{\url{https://www.lmfdb.org/api/g2c_curves/?mw_rank_proved=0}}}
  $69$ genus~$2$ curves for which the rank of the Jacobian was only
  known conditionally.  In $45$ of these cases (the first being that
  in Example~\ref{ex:firstLMFDB}) we have
  $\Sha(\J/\Q)[2^\infty] \isom (\Z/2\Z)^2$.  These are therefore
  included in the $4088$ cases analysed above.  These are all cases
  where the canonical element of $S^{(2)}(\J/\Q)$ is trivial, there
  are no deficient places, and $\rank \, \J(\Q) = 0$ or $1$ (the
  latter only for {\tt 776185.a.776185.1}).
The remaining $24$ cases comprise all but $2$ curves in list $C$, and
all curves in list $D$. There are no curves in lists $A$ and $B$ since
our methods in these cases are already used by Stoll's Magma function
{\tt MordellWeilGroupGenus2}. The two curves in list $C$ that were
already resolved are {\tt 38267.a.38267.1} and {\tt
  523925.a.523925.1}. These are isogenous to curves where
Lemma~\ref{lem:PS} and Remark~\ref{rem:creutz} already give sufficient
information about the pairing to conclude.
\end{Remark}


\begin{thebibliography}{MM}

\frenchspacing
\renewcommand{\baselinestretch}{1}

\bibitem[BG]{BG}
M.~Bhargava and B.H.~Gross,  
The average size of the 2-Selmer group of Jacobians of
hyperelliptic curves having a rational Weierstrass point,
in {\em Automorphic representations and $L$-functions},
D.~Prasad, C.S.~Rajan, A.~Sankaranarayanan and J.~Sengupta (eds),
23--91,
Tata Inst. Fundam. Res. Stud. Math., {\bf 22}, Tata Inst. Fund. Res.,
Mumbai, 2013.

\bibitem[BL]{BirkenhakeLange}
C. Birkenhake and H. Lange,  
{\em Complex abelian varieties}, Second edition,
Grundlehren der mathematischen Wissenschaften,
{\bf 302}, Springer-Verlag, Berlin, 2004.

\bibitem[BSSVY]{lmfdb-paper}
A.R.~Booker, J.~Sijsling, A.V.~Sutherland, J.~Voight, and D.~Yasaki,
A database of genus 2 curves over the rational numbers,
Twelfth Algorithmic Number Theory Symposium (ANTS XII),
{\em LMS Journal of Computation and Mathematics} {\bf{19}} (2016), 235--254.

\bibitem[BCP]{magma}
W. Bosma, J. Cannon and C. Playoust, 
The Magma algebra system I: The user language, {\em J. Symb. Comb.} {\bf{24}}, 
235-265 (1997), \url{http://magma.maths.usyd.edu.au/magma/} 

\bibitem[BF]{BruinFlynn}
N. Bruin and E.V. Flynn,  
Exhibiting SHA[2] on hyperelliptic Jacobians,
{\em J. Number Theory} {\bf 118} (2006), no. 2, 266--291.

\bibitem[BS]{BruinStoll}
N. Bruin and M. Stoll,  
Deciding existence of rational points on curves: an experiment,
{\em Experiment. Math.} {\bf 17} (2008), no. 2, 181--189.

\bibitem[C1]{CasselsIII}
J.W.S. Cassels,
Arithmetic on curves of genus 1, III. The Tate-Šafarevič and Selmer groups,
{\em Proc. London Math. Soc.} (3) {\bf 12} (1962), 259--296.

\bibitem[C2]{CasselsIV}
J.W.S. Cassels,  
Arithmetic on curves of genus 1,
IV. Proof of the Hauptvermutung.
{\em J. reine angew. Math.} {\bf 211} (1962), 95--112.

\bibitem[C3]{Cassels-RCF}
J.W.S. Cassels,
{\em Rational quadratic forms},
London Mathematical Society Monographs, {\bf{13}}, Academic Press,
London-New York, 1978.

\bibitem[C4]{Cassels-G2}
J.W.S. Cassels,
The Mordell-Weil group of curves of genus 2, in
{\em Arithmetic and geometry, Vol. I}, M.~Artin and J~Tate (eds), 27--60,
Progr. Math., 35, Birkhäuser, Boston, Mass., 1983.

\bibitem[C5]{Ca98}
J.W.S. Cassels, 
Second descents for elliptic curves,
{\em J. reine angew. Math.} {\bf 494} (1998), 101--127.

\bibitem[CF]{CF}
J.W.S. Cassels and E.V. Flynn,  
{\em Prolegomena to a middlebrow arithmetic of curves of genus 2},
London Mathematical Society Lecture Note Series, {\bf 230},
Cambridge University Press, Cambridge, 1996.

\bibitem[Cl]{ob}
P.L. Clark,
{\em The period-index problem in WC-groups II: abelian varieties},
preprint 2004, \url{arXiv:math/0406135 [math.NT]}
    
\bibitem[CFS]{minred234}
J.E. Cremona, T.A. Fisher and M. Stoll,
Minimisation and reduction of 2-, 3- and 4-coverings of elliptic curves,
{\em Algebra \& Number Theory} {\bf 4} (2010), no. 6, 763--820.

\bibitem[Cr]{Creutz}
B. Creutz, 
Improved rank bounds from 2-descent on hyperelliptic Jacobians,
{\em Int. J. Number Theory} {\bf 14} (2018), no. 6, 1709--1713.

\bibitem[D]{Donagi}
R. Donagi, Group law on the intersection of two quadrics,
{\em Ann. Scuola Norm. Sup. Pisa Cl. Sci. (4)} {\bf 7} (1980),
no. 2, 217--239.

\bibitem[Do]{Donnelly}
S. Donnelly, {\em Algorithms for the Cassels-Tate pairing}, preprint, 2015.

\bibitem[Fi1]{testeqtc}
T.A. Fisher, Testing equivalence of ternary cubics, in {\em
Algorithmic number theory}, 333–345, Lecture Notes in
Comput. Sci., {\bf 4076}, Springer, Berlin, 2006.

\bibitem[Fi2]{bq-ctp}
T.A. Fisher,
On binary quartics and the Cassels-Tate pairing,
{\em Res. Number Theory} {\bf 8} (2022), no. 4, Paper No. 74, 13 pp.

\bibitem[FL]{FL}
T.A. Fisher and M. Liu,
{\em Minimisation of $2$-coverings of genus $2$ Jacobians},
in preparation.

\bibitem[Fl]{Flach}
M. Flach,  
A generalisation of the Cassels-Tate pairing,
{\em J. reine angew. Math.} {\bf 412} (1990), 113--127.

\bibitem[F1]{72}
E.V. Flynn,  
The Jacobian and formal group of a curve of genus 2 over an
arbitrary ground field,
{\em Math. Proc. Cambridge Philos. Soc.} {\bf 107} (1990), no. 3, 425--441.

\bibitem[F2]{biquadratic}
E.V. Flynn,  
The group law on the Jacobian of a curve of genus 2,
{\em J. reine angew. Math.} {\bf 439} (1993), 45--69.

\bibitem[FPS]{FPS}
E.V. Flynn, B. Poonen and E.F. Schaefer,
Cycles of quadratic polynomials and rational points on a genus-2 curve,
{\em Duke Math. J.} {\bf 90} (1997), no. 3, 435--463.

\bibitem[FTvL]{2-coverings}
E.V. Flynn, D. Testa and R. van Luijk,  
Two-coverings of Jacobians of curves of genus~2,
{\em Proc. Lond. Math. Soc.} (3) {\bf 104} (2012), no. 2, 387--429.

\bibitem[FH]{FH}
W. Fulton and J. Harris,  
{\em Representation theory,
  A first course}, Graduate Texts in Mathematics, {\bf 129},
Springer-Verlag, New York, 1991.

\bibitem[GG]{GG}
D.M. Gordon and D. Grant,
Computing the Mordell-Weil rank of Jacobians of curves of genus two,
{\em Trans. Amer. Math. Soc.} {\bf 337} (1993), no. 2, 807--824.

\bibitem[HS]{HS}
Y.~Harpaz and A.N.~Skorobogatov,  
Hasse principle for Kummer varieties,
{\em Algebra \& Number Theory} {\bf 10} (2016), no. 4, 813--841.

\bibitem[LMFDB]{lmfdb}
The LMFDB Collaboration, The $L$-functions and Modular Forms Database,
(online; accessed May 2023), \url{https://www.lmfdb.org/}

\bibitem[L]{Lichtenbaum-dualitythm}
S. Lichtenbaum,  
Duality theorems for curves over $p$-adic fields,
{\em Invent. Math.} {\bf 7} (1969), 120--136.

\bibitem[LvL]{Logan-vanLuijk}
A. Logan and R. van Luijk,  
Nontrivial elements of Sha explained through K3 surfaces,
{\em Math. Comp.} {\bf{78}} (2009), no. 265, 441--483.

\bibitem[Mi]{Milne}
J.S. Milne, {\em Arithmetic duality theorems},
Second edition, BookSurge, LLC, Charleston, SC, 2006.

\bibitem[Mo]{Morgan}
A. Morgan,
{\em Hasse principle for Kummer varieties in the case of generic 2-torsion},
in preparation.
  
\bibitem[NR]{NR}
M.S. Narasimhan and S. Ramanan,
Moduli of vector bundles on a compact Riemann surface,
{\em Ann. of Math.} (2) {\bf 89} (1969), 14--51. 

\bibitem[N]{Newstead}
P.E. Newstead,
Stable bundles of rank 2 and odd degree over a curve of genus 2,
{\em Topology} {\bf 7} (1968), 205--215.

\bibitem[PSc]{PoonenSchaefer}
B. Poonen and E.F. Schaefer,
Explicit descent for Jacobians of cyclic covers of the projective line,
{\em J. reine angew. Math.} {\bf 488} (1997), 141--188.

\bibitem[PSt]{PS}
B. Poonen and M. Stoll,  
The Cassels-Tate pairing on polarized abelian varieties,
{\em Ann. of Math.} (2) {\bf 150} (1999), no. 3, 1109--1149.

\bibitem[R]{Reid}
M. Reid, {\em The complete intersection of two or more quadrics},
PhD thesis, University of Cambridge, 1972.

\bibitem[Se]{Serre-LF}
J.-P. Serre,
{\em Local fields}, Graduate Texts in Mathematics, {\bf 67},
Springer-Verlag, New York-Berlin, 1979.

\bibitem[SW]{SW}
A. Shankar and X. Wang,
Rational points on hyperelliptic curves having a marked
non-Weierstrass point,
{\em Compos. Math.} {\bf 154} (2018), no. 1, 188--222.

\bibitem[Si]{Simon}
D. Simon,
{\em Quadratic equations in dimensions 4, 5 and more},
preprint, 2005, \url{http://www.math.unicaen.fr/~simon/}

\bibitem[St]{Stoll}
M. Stoll,
Implementing 2-descent for Jacobians of hyperelliptic curves,
{\em Acta Arith.} {\bf 98} (2001), no. 3, 245--277.

\bibitem[SvL]{unfaking}
M. Stoll and R. van Luijk,  
Explicit Selmer groups for cyclic covers of $\PP^1$,
{\em Acta Arith.} {\bf 159} (2013), no. 2, 133--148.

\bibitem[Sw]{brief}
H.P.F. Swinnerton-Dyer,  
{\em A brief guide to algebraic number theory},
London Mathematical Society Student Texts, {\bf 50},
Cambridge University Press, Cambridge, 2001.

\bibitem[Ta]{Tate}
J. Tate,  
Duality theorems in Galois cohomology over number fields,
{\em Proc. Internat. Congr. Mathematicians (Stockholm, 1962)}
pp. 288--295, Inst. Mittag-Leffler, Djursholm, 1963.

\bibitem[T1]{Thorne}
J.A. Thorne,  
A remark on the arithmetic invariant theory of hyperelliptic curves,
{\em Math. Res. Lett.} {\bf 21} (2014), no. 6, 1451--1464.

\bibitem[T2]{thorne-red}
J.A. Thorne,  
{\em Reduction theory for stably graded Lie algebras}, in preparation.

\bibitem[Wa]{Wang}
X. Wang, 
Maximal linear spaces contained in the based loci of pencils of quadrics,
{\em Algebr. Geom.} {\bf 5} (2018), no. 3, 359--397.

\bibitem[We]{Weil-Hermite}
A. Weil,  
Remarques sur un mémoire d'Hermite,
{\em Arch. Math. (Basel)} {\bf 5} (1954), 197--202.

\bibitem[Y1]{JialiThesis}
J. Yan,
{\em Computing the Cassels-Tate pairing for Jacobian varieties
of genus two curves}, PhD thesis, University of Cambridge, 2021,
\url{https://doi.org/10.17863/CAM.72729}

\bibitem[Y2]{Jiali-rat2tors}
J. Yan,
{\em Computing the Cassels-Tate pairing for genus two Jacobians with
 rational two torsion points},
preprint, 2021,
\url{https://arxiv.org/abs/2109.08258}

\end{thebibliography}
\end{document}